\documentclass[12pt]{amsart}
\usepackage[english]{babel}

\input{packages.sty}
\input{commands.sty}

\title{On perfectoid ideals}
\author[L. Navarro Chafloque]{Léo Navarro Chafloque}
\address{EPFL SB MATH CAG,
MA C3 585 (Bâtiment MA),
Station 8,
1015 Lausanne}
\email{leo.navarrochafloque@epfl.ch}



\begin{document}

\begin{abstract}
In this paper, we study perfectoid ideals using the reduction modulo $p$ of perfectoid rings. We show that perfectoid ideals correspond to ideals in the reduction modulo $p$ on which Frobenius is surjective. We establish the existence of universal $p$-adic power series whose evaluations, together with their compatible $p$-power roots, generate perfectoidizations of ideals. In the special case of the perfectoidization of $\Z_p^{\cycl}\abracket{t^{1/p^\infty}}/(t-1)$, we describe a linear algebra algorithm for computing finite approximations of such generators. We illustrate this by an explicit approximation modulo $8$ of an element $\sigma\in (t-1)_{\pfd}\subseteq \Z_2^{\cycl}\abracket{t^{1/2^\infty}}$ which admits compatible sequences of $2$-power roots and satisfies $\sigma\equiv t-1 \mod 2$.
\end{abstract}

\maketitle  

\tableofcontents

\section{Introduction}

One of the most surprising aspects of the theory of perfectoid algebras is that if $R$ is a given perfectoid ring and $I\subseteq R$ is an ideal, then there is a \emph{surjective} map to a perfectoid ring $R\to (R/I)_{\pfd}$ where the map is initial among maps from $R$ to perfectoid rings that send $I$ to zero. The surprising aspect is not the existence, but the surjectivity of this map, even in the case where $I$ is not generated by elements which admit a compatible sequence of $p$-power roots.\footnote{Actually, in the early days of the theory, this was believed to be false in the case of \Cref{equation:zp1}, as the now superseded \cite[Remark 2.2.4]{Scholze2015} indicates (see \cite[Remark 7.5]{bhatt2022Prismsa}).} This is proven in \cite[Theorem 7.4]{bhatt2022Prismsa} as a consequence of \cite[Theorem 7.14]{bhatt2022Prismsa}. This leads to the fact that for every ideal $I\subseteq R$ there is a notion of \emph{perfectoidization} of the ideal.

A first interesting non-trivial example is the quotient
\begin{equation}\label{equation:zp1}
\Z_p^{\cycl}\abracket{t^{1/p^\infty}}/(t-1).
\end{equation}
With $p$ inverted, this example was actually studied already back in the article \cite[Section 1]{BSMF_1978__106__225_0}, although in a different language. Translating what they did into modern language, one can deduce that what they showed implies that $(t-1)_{\pfd}$ is generated by the $p$-adic closure of the nilpotents in the above ring. 

Without their techniques, one can actually show that the ideal $(t-1)_{\pfd}$ is ($p$-adically) generated by elements which admit a compatible sequence of $p$-power roots (this, in fact, characterizes perfectoid ideals, see \cite[Lemma 4.6]{ishizuka2026decompositiontheoremperfectoidrings}). Even more interestingly, one of the consequences of the present work is that there exists an expression
\begin{equation}\label{equation:expression_power}
t-1=\sum_{n=0}^{\infty} f_n p^n,
\end{equation}

where $f_n\in \Z_p^{\cycl}\abracket{t^{1/p^\infty}}$ are elements which admit a compatible sequence of $p$-power roots; for each $n\in \N$, we have $f_n\in (t-1)_{\pfd}$, and actually $(t-1)_{\pfd}=\overline{(f_n^{1/p^{\infty}})_{n\in \N}}$ is $p$-adically generated by the $f_n$'s and their $p$-power roots. In particular, there is an element $\sigma=f_0$ which admits a compatible sequence of $p$-power roots, with $\sigma\in (t-1)_{\pfd}$ and $\sigma\equiv t-1 \mod p$. As explained in \Cref{prop:no_finite_approximation}, the element $\sigma$ is necessarily a non-trivial convergent power series in the $p$-adic topology -- the author notes that this is a concrete example which highlights the importance of $p$-completion in the theory.

The first impulse of this work was to see if it was possible to understand such an element as $\sigma$ mentioned above. To approach this question, we noted the following: it is well known that the category of perfectoid algebras over a perfectoid ring is equivalent to a category of complete perfect algebras over a perfect complete ring using the tilt construction. Actually, this equivalence works also if we simply consider \emph{the reduction modulo p}, instead of the tilt. This is, for example, observed in \cite[Proposition 3.2.6]{bhatt2025Aspects} for the derived reduction modulo $p$ (which, in the context of $p$-torsion-free perfectoid rings, agrees with the usual reduction modulo $p$) -- we show in this work that the classical reduction modulo $p$ actually always works to construct such an equivalence. This is the content of \Cref{proposition:eidtilt}. We identify the target category among $\F_p$-algebras and give it a name for the sake of clarity in the article: \emph{eidetic perfectoid rings} (\Cref{definition:eid}).

Going back to the matter of perfectoid ideals which initially interested us, the idea was to use the fact that perfectoid rings were totally determined by their reduction modulo $p$, where it was potentially easier to make explicit calculations (mainly because the freshman's dream identity $(a+b)^p=a^p+b^p$ holds modulo $p$). An important aspect was also that when taking the reduction modulo $p$, in contrast with the tilt construction, \emph{no completeness} was required in the theory. As $p$-adic completion is one of the things which sometimes make things difficult to calculate explicitly, this new category without any topology seemed interesting to the author to investigate. As an example, the reduction in this context of the perfectoid ring $\Z_p^{\cycl}\abracket{t^{1/p^\infty}}$ can be identified with
\[\F_p[x^{1/p^\infty},t^{1/p^\infty}]/(x).\]
Note the presence of infinitely many nilpotent elements in this ring: keeping these nilpotents in the theory is actually what makes possible the reconstruction of the initial perfectoid ring from its reduction modulo $p$.

As a first step in this direction of understanding perfectoid ideals using this reduction, we show in this work that there is a correspondence between perfectoid ideals and so-called \emph{eidetic perfectoid ideals} in the eidetic reduction (\Cref{definition:eid-ideals}) which are simply ideals on which the absolute Frobenius acts surjectively. The following is the main result in \Cref{section:pfdideals_modp}.

\begin{thm}[\Cref{theorem:eid-tilt-ideals}]\label{thm:perfectoid-ideals-correspondence}
Let $R$ be a perfectoid ring and $S=R/\varpi^p R$ the associated eidetic perfectoid ring (see \Cref{proposition:eidtilt}). Then reduction modulo $\varpi^p$ induces a bijection
\[\{J\subseteq  R \mid J \text{ is a perfectoid ideal}\}\to \{I\subseteq  S \mid I\text{ is an eidetic perfectoid ideal}\}.\]
\end{thm}

This is a correspondence of the form that appeared in \cite[Main Theorem A]{Dine2024}, and our account actually reproves this correspondence and slightly generalizes it (see \Cref{remark:dine_isihizuka}).

Another question is how to get the expression in \Cref{equation:expression_power}. In \Cref{section:on_equations} we show, using the theory of perfect prismatic envelopes (\Cref{section:prism-env}), that there are actually \emph{universal equations, in the form of power series} that, given an element $r\in R$ in a perfectoid ring, allow us to produce $p$-adic generators of $(r)_{\pfd}$ along with an expression as in \Cref{equation:expression_power}. Namely, we show the following.

\begin{proposition}[\Cref{prop:UniversalEquations}, \Cref{remark:over_an_algebra}]\label{prop:univ_eq_intro}
Fix a base perfectoid ring $R_0$ (for example $\Z_p^{\cycl}$ or the $p$-completion of $\Z_p[p^{1/p^{\infty}}]$). Then there exist power series
\[g_n\in R_0\abracket{t_j^{1/p^{\infty}}}_{j\in \N}\]
that admit a compatible sequence of $p$-power roots in the power series ring, such that for any perfectoid $R_0$-algebra $R$ and any element $r\in R$, with a chosen representation as $r=\sum_{i=0}^{\infty}\sharp(r_i^{1/p^i})p^i$, we have $$(r)_{\pfd}=\overline{(g_n(r_j)^{1/p^{\infty}})_{n\in \N}}.$$
Moreover,
\[r=\sum_{n=0}^{\infty}g_n(r_j)^{1/p^n}p^n.\]
\end{proposition}
These power series $g_n$ are not unique at all, and not canonical in any sense, but they provide for every $R_0$-algebra the ``same way'' to construct the perfectoidization of an ideal: take generators of the ideal, and then apply these $p$-adic power series to get elements with a system of compatible $p$-power roots which generate $p$-adically the perfectoidization of the ideal.

In the last section of this work, \Cref{section:program}, we go back to the original example \Cref{equation:zp1}. Namely, using the eidetic reduction described in \Cref{section:pfdideals_modp}, we explain how one can, in principle, by solving linear systems over $\F_p$ (\Cref{section:algori}), get approximations of an element $f_0=\sigma$ as in \Cref{equation:expression_power}. Namely, using a SageMath script\footnote{The SageMath code is available at this \href{https://github.com/leo-nc-git/perfectoid_kernel}{GitHub repository}.} implementing the situation, we are able to get the following result.

\begin{proposition}[\Cref{prop:kernel_results_2_2}]\label{prop:kernel_results_intro}
For $p=2$, there is an element $\sigma=f_0$ as in \Cref{equation:expression_power} which satisfies
\[
\begin{aligned}
\sigma \equiv{}&t-1\\
&+2t^{1/8}\Bigl(1-\xi_{16}^2-\xi_{16}^4-\xi_{16}^6
 +(1+\xi_{16}^2+\xi_{16}^4+\xi_{16}^6)t^{1/4}\\
&\qquad +(1+\xi_{16}^2-\xi_{16}^4+\xi_{16}^6)t^{1/2}
 +(1-\xi_{16}^2+\xi_{16}^4-\xi_{16}^6)t^{3/4}\Bigr)\\
&+4(1+t^{1/2})\Bigl(\xi_{16}^2+\xi_{16}^3+\xi_{16}^4+\xi_{16}^7
 +(\xi_{16}^2+\xi_{16}^4+\xi_{16}^5+\xi_{16}^7)t^{1/16}\\
&\qquad +(\xi_{16}+\xi_{16}^5)(t^{1/8}+t^{3/8})
 +(1+\xi_{16}^2+\xi_{16}^5+\xi_{16}^7)t^{3/16}\\
&\qquad +(\xi_{16}^3+\xi_{16}^4+\xi_{16}^6+\xi_{16}^7)t^{1/4}
 +(\xi_{16}+\xi_{16}^2+\xi_{16}^3+\xi_{16}^4)t^{5/16}\\
&\qquad +(1+\xi_{16}+\xi_{16}^2+\xi_{16}^3)t^{7/16}\Bigr)\pmod{8}.
\end{aligned}
\]
Recall that this element $\sigma$ admits a compatible sequence of $2$-power roots, $\sigma\in (t-1)_{\pfd}$ and $\sigma\equiv t-1 \mod 2$. In the above, $\xi_{16}$ denotes a primitive sixteenth root of unity.

\end{proposition}

The calculations of the linear systems involved in determining $\sigma$ grow exponentially in complexity; see \Cref{remark:github}, \Cref{remark:number_of_solutions} and \Cref{remark:other_results} for an account.

\begin{notation}\label{notation:setup} Here is some notation that we use throughout the article.
\begin{itemize}
    \item We fix a prime number $p\in \N$.
    % \item When $R$ is a perfectoid ring, $\varpi\in R$ is an element such as in \Cref{notation:varpi}.
    \item If $R$ is a perfectoid ring (\Cref{def:pfd}), we denote by $\Pfd_R$ the category of $\varpi$-complete perfectoid $R$-algebras, where $\varpi$ is an element as in \Cref{notation:varpi}. Whenever we consider a perfectoid ring $R$, we implicitly choose such a $\varpi\in R$.
    \item If $S$ is a ring such that $pS=0$, then we denote by $F_S\colon S\to S$ the absolute Frobenius map.
    \item If $A$ is any ring and $a\in A$, then $A[a]=\{b \in A\mid ba=0\}$.
    \item When we write $\abracket{-}$, it means that we consider algebras implicitly completed in a topology that should be clear from context or that we will specify on the spot. The notation $\overline{(-)}$ stands for the closure of a subset in the implicit topology.
    \item We occasionally use (mostly for convenience) the language of animated rings (see, for example, \cite[Section 5.1.1]{cesnavicius2024purity}).
    \item $\xi_n$ denotes a primitive $n$-th root of unity.
\end{itemize}

\end{notation}

\subsection*{Acknowledgements} The author wants to thank Ryo Ishizuka for general discussions on the topic of the paper. The author also wants to thank Pierre-Louis Frabel for discussing the subject with him. The author was supported by FNS grant \#200021-231484.

\subsubsection*{AI disclosure} The SageMath software used to compute the results in \Cref{prop:kernel_results_2_2} was written with the help of GitHub Copilot, GPT-5.5 Sol and GPT-6.1 Sol. At the final stage of the paper, the author asked GPT-6.1 Sol to review the paper as a mathematical assistant; it found some imprecisions and asked for clarifications, which the author then addressed.

\section{Recollections on perfectoid rings and ideals}

Recall the following definition.

\begin{definition}\label{def:pfd} A ring $R$ is \emph{perfectoid} if it contains an element $\varpi\in R$ such that $\varpi^p \mid p$, the ring is $\varpi$-complete and Fontaine's $\theta$ map (\cite[Section 2.1]{FarguesFontaine2018})
\[\theta\colon A_{\inf}(R) \to R\]
is surjective with a principal kernel.
\end{definition}

\begin{notation}\label{notation:varpi} For any perfectoid ring $R$, an element $\varpi\in R$ will be implicitly taken to have the properties in \Cref{def:pfd}. Recall that, up to multiplication by a unit in $R$, any such element admits a compatible system of $p$-power roots (\cite[Lemma 3.9]{BMS18}). We usually assume that such a system has been chosen and that $\varpi$ is the first term of this sequence; in this case, we denote the system of compatible $p$-power roots of such an element by $\varpi^{\flat}\in R^{\flat}$. Note that the ideal $\varpi R$ is unambiguously defined.
Also, recall that we can always choose $\varpi$ (possibly generating a different ideal) such that $\varpi^p=pu$, where $u\in R$ is a unit. Therefore, in this article, one can always read $\varpi^p R$ as $pR$ in this important special case.
We also choose an element $d\in A_{\inf}(R)$ that generates the kernel of Fontaine's map $\theta_R\colon A_{\inf}(R)\to R$. We can choose $d$ such that $d=[\varpi^{\flat p}]b-p$ for $b\in A_{\inf}(R)$ and $\varpi^{\flat }\in R^{\flat}$; if $\varpi^p=pu$ for a unit $u\in R$, then $b$ can be chosen to be a unit.
    
\end{notation}

We take the opportunity to advertise the following alternative definition of perfectoid rings (\Cref{prop:pfd_Derived_Def}), which, at least to the author's knowledge, does not appear in the literature in this exact form involving derived quotients. Although the content of the statement appears in
\cite[Corollary 16.3.75]{gabber2018foundationsringtheory}, the nature of the proof there is different. We need the following technical lemma for the proof -- this \Cref{lemma:principal_miracle} is also a key tool for \Cref{proposition:eidtilt}, and was the original motivation for its formulation.

\begin{lemma}\label{lemma:principal_miracle}Let $T$ be an $\F_p$-algebra with an element $a\in T$ with the property that $a^p=0$ and
\[(-)^p\colon T/aT\to T\]
is an isomorphism. Let $a^{\flat}$ be any preimage of $a$ under $T^{\flat}\to T$. Then the kernel of the surjective map $T^{\flat}\to T$ is the principal ideal $a^{\flat p}T^{\flat}$.
\end{lemma}

\begin{proof}
Recall that the map $T^{\flat}\to T$ is surjective because $T$ is semiperfect. 
Denote $K=\ker(T^{\flat}\to T)$. If $y=(y_n)\in K$, it simply means that $y_0=0$. We show below the non-trivial inclusion $K\subset a^{\flat p}T^{\flat}$.

We show by induction on $N\geq 1$ that there exist elements $x_n,z_n\in T^{\flat}$ with the property that
\[y=a^{\flat p}\left(\sum_{n=1}^N(a^{\flat})^{p^{n}-p}x_n\right)+z_{N}\] 
and $(z_N)_N=0$. Here, the notation $(-)_N$ applied to elements of $T^{\flat}$ denotes taking the $N$-th coordinate.

\begin{itemize}
    \item For $N=1$, as $y_1^p=0$, we have $y_1=a\sigma$ for some $\sigma\in T$. Let $x_1$ be any lift of $\sigma$ in $T^{\flat}$. Define $z_1:=y-a^{\flat p}x_1$. By construction, $(z_1)_1=y_1-a\sigma=0$.
    \item  Now we conduct the induction step. As $(z_N)_N=0$, we have $z_N^{1/p^N}\in K$. Therefore, we can treat it as $y$ in the preceding argument to get some $\widetilde{x}_{N+1},\widetilde{z}_{N+1}\in T^{\flat}$ with $z_N^{1/p^N}=a^{\flat p}\widetilde{x}_{N+1}+\widetilde{z}_{N+1}$ with $(\widetilde{z}_{N+1})_1=0$. We take this to the power $p^N$ to get, with $x_{N+1}=\widetilde{x}_{N+1}^{p^N}$ and $z_{N+1}=\widetilde{z}_{N+1}^{p^N}$
    $$z_N=a^{\flat p^{N+1}}x_{N+1}+z_{N+1}.$$
    Note that $(z_{N+1})_{N+1}=0$ by construction.
\end{itemize}

Now note that both sequences $\left(\sum_{n=1}^N(a^{\flat})^{p^{n}-p}x_n\right)_{N\geq 1}$ and $(z_n)_{n\geq 1}$ are Cauchy for the inverse limit topology
\footnote{A sequence $(s_n)_{n\geq 1}$ of elements of $T^{\flat}$ converges to zero in the inverse limit topology if for any $M\in \N$ there exists an $L\in\N$ with $(s_{n})_M=0$ for any $n\geq L$. Note that then $(s_n)_{M'}=0$ for any $0\leq M'\leq M$ as $(s_n)_{M'}=(s_n)_{M}^{p^{M-M'}}$. Note therefore that by induction, a Cauchy sequence in this topology defines a unique element of the tilt, showing that the tilt is complete with respect to this uniform structure.} on $T^{\flat}$ with $\lim_{n\to \infty}z_n=0$. If $\alpha:=\sum_{n=1}^{\infty}(a^{\flat})^{p^{n}-p}x_n$, then $y=a^{\flat p}\alpha$, which proves the desideratum.
% % Namely, as $$(-)^p\colon R/\varpi R\to R/\varpi^{p}R$$ induces an isomorphism $$(-)^p\colon R/\varpi^{1/p^n}\to R/\varpi^{1/p^{n-1}}R$$ for every $n\in \N$, the above claim follows.

% Note that $$(-)^p\colon R/\varpi R\to R/\varpi^{p}R$$ is an isomorphism, there is an induced isomorphism $$(-)^p\colon R/\varpi^{1/p}\to R/\varpi R.$$ Therefore, we have $y_1=\varpi^{1/p}\sigma$ for some $\sigma \in R/\varpi R$. Lift $\sigma$ to a $x_1\in R^{\flat}$ along the surjection $R^{\flat}\to R/\varpi R$. Now $(y-\varpi^{\flat}x_1)_1=0$. Therefore, if $z_1:=y-\varpi^{\flat}x_1$ we have that $z_1^{1/p}\in K$. Applyng again the exact same process gives by induction elements $x_n,z_n\in R^{\flat}$ for $n\in \N$ such that for any $N\in \N$ we have

% with the property that $(z_N)_N=0$. Note that both $\left(\sum_{n=1}^N\varpi^{\flat p^n -1}x_n\right)$ and $(z_n)$ converges in the inverse limit topology on $R^{\flat}$ with $\lim_{n\to \infty}z_n=0$. If $\alpha=\sum_{n=0}^{\infty}x_n\varpi^{\flat p^n -1}$ then $y=\varpi^{\flat}\alpha$, concluding the proof as desired.

\end{proof}

\begin{proposition}[Equivalent definition of perfectoids]\label{prop:pfd_Derived_Def} Let $R$ be a ring. Then the following conditions are equivalent.
\begin{enumerate}
    \item $R$ is perfectoid.
    \item \begin{enumerate}
        \item[(i)] There exists $\varpi\in R$ such that $\varpi^p \mid p$ and 
         \[(-)^p\colon R/^L\varpi R\to R/^L\varpi^p R\]
    induces an isomorphism of animated rings on derived quotients,
    \item[(ii)] and $R$ is $\varpi$-(derived) complete.
    \end{enumerate}
\end{enumerate}
    
\end{proposition}
\begin{proof}Note that derived $\varpi$-completeness and classical $\varpi$-completeness are the same in the context of this proposition because of the bounded $\varpi$-torsion of $R$, which we will see holds.
    
    To check that this map of animated rings is an isomorphism, it suffices to check that it is an isomorphism on $\pi_1$ and $\pi_0$, as the rings are $1$-truncated. Here this translates to

\begin{enumerate}
    \item[\textit{(a)}] $(-)^p\colon R/\varpi R\to R/\varpi^p R$
    \item[\textit{(b)}] $(-)^p\colon R[\varpi]\to R[\varpi^p]$
\end{enumerate}
being isomorphisms -- see \Cref{remark:derived_p_power} for a precise account of the matter. For the implication from (1) to (2), we refer to \cite[Lemma 3.10]{BMS18} for $(a)$ and to the end of \cite[Section 2.1.2]{cesnavicius2024purity} for $(b)$.

We now prove the converse direction. Note that from (a) we get the surjectivity of the Frobenius on $R/pR$. We can then apply \cite[Lemma 3.9]{BMS18} to replace $\varpi$ by a unit multiple admitting compatible $p$-power roots that we denote by $\varpi^{\flat}\in R^{\flat}$; note that the statements are insensitive to this change.

First, we claim that if $(a)$ holds then the untilt map
\begin{equation*}
R^{\flat}/\varpi^{\flat}R^{\flat}\to R/\varpi R
\end{equation*}

is an isomorphism. Showing that $R^{\flat}/\varpi^{\flat} R^{\flat}\to R/\varpi R$ is an isomorphism is exactly showing that the kernel of the projection $\lim_{x\mapsto x^p}R/\varpi R=R^{\flat}\to R/\varpi R$ is the principal ideal $\varpi^{\flat}R^{\flat}$. To this end, note that \Cref{lemma:principal_miracle} implies the desideratum for $T=R/\varpi R$ and $a=\varpi^{1/p}$.

% % Also, if $(b)$ holds, we have for every $n\in \Z$ that
% \[R[\varpi^{1/p^n}]=R[\varpi].\]
We now explore the consequences of (b).
Say that for some $m\geq 2$ we have $\varpi^m x=0$ for some $x\in R$. Then $\varpi\varpi^{m-1}x=0$. Then $\varpi^{m-1}x\in R[\varpi]$. But also $(\varpi^{m-1}x)^p=0$ as $p\geq 2$. So by the injectivity of $R[\varpi]\to R[\varpi^p]$, we have $\varpi^{m-1}x=0$. Now we can repeat the argument until we reach $\varpi x=0$. Therefore, $(b)$ can now be understood as saying, if $R[\varpi^{\infty}]:=\bigcup_{m\geq1}R[\varpi^m]$, that
\[R[\varpi^{\infty}]\xrightarrow{(-)^p}R[\varpi^{\infty}]\]
is an isomorphism.

We now claim that the following untilt map (which in this case is a morphism of $p$-torsion abelian groups)
\[R^{\flat}[\varpi^{\flat}]\to R[\varpi]\]
is an isomorphism. Namely, we can see that it is injective: if $(x_n)\in R^{\flat}[\varpi^{\flat}]$ is sent to zero, it means that $x_0=0$. But then $x_{n}^{p^n}=0$ and, as $x_n\in R[\varpi]$, we have $x_n=0$ by hypothesis. That the map is surjective is also a direct consequence of the surjectivity in $(b)$ and the induced surjectivity for every $n\in \N$ of $(-)^p\colon R/\varpi^{1/p^n}\to R/\varpi^{1/p^{n-1}}R$.

Now we can conclude: namely, take an element $d$ in the kernel of $A_{\inf}(R)\to R$ as in \Cref{notation:varpi}. Namely, as $R/pR$ is semiperfect, we know that this map is surjective. If $\varpi^p\beta=p$, take a lift $b\in A_{\inf}(R)$ of $\beta$. Let $d=[\varpi^{\flat p}]b-p$. This element is distinguished (to see this, perform the computation of the difference in $W_2$ and use the $\varpi^{\flat}$-completeness of $R^{\flat}$) and as $R^{\flat}$ is $\varpi^{\flat}$-complete, the ring $A_{\inf}(R)/d$ is perfectoid by \cite[Equation 2.1.2.1]{cesnavicius2024purity}.

Now we want to argue that
\begin{equation}\label{equation:Ainfcomparision}
    A_{\inf}(R)/(d)\to R
\end{equation}
is an isomorphism. Note that the tilt of the perfectoid ring $A_{\inf}(R)/(d)$ is $R^{\flat}$. So the $\varpi$-torsion of this ring is identified with $R^{\flat}[\varpi^{\flat}]$ by the end of \cite[Section 2.1.2]{cesnavicius2024purity}.

Note that in the map of \Cref{equation:Ainfcomparision}, both rings are $[\varpi^{\flat}]$-complete (seen as $A_{\inf}(R)$-algebras, say) and therefore we can check that the map is an isomorphism using derived Nakayama. But on $\pi_0$ the derived quotient by $[\varpi^{\flat}]$ is 
\[R^{\flat}/\varpi^{\flat}R^{\flat}\to R/\varpi R\]
which we showed first was an isomorphism and on $\pi_1$ of the derived quotient we get 
\[R^{\flat}[\varpi^{\flat}]\to R[\varpi]\]
which we also showed was an isomorphism. 
\end{proof}

\begin{rem}\label{remark:derived_p_power}In the context of \Cref{prop:pfd_Derived_Def}, we explain precisely the map
 \[(-)^p\colon R/^L\varpi R\to R/^L\varpi^p R\]
 on animated quotients, and its functoriality. Note that \emph{this is not} the absolute Frobenius on animated $\F_p$-algebras (the target is not the same as the source). Let 
 $$D=\Z[t,s]/(t^ps-p).$$
 Note that $t$ is a non-zero-divisor in this ring. Note that choosing a $\varpi\in R$ with $\varpi^p$ dividing $p$ gives a map $D\to R$ sending $t\mapsto \varpi$ and $s$ to an appropriate element in $R$. We fix such a $D$-algebra structure. We consider $\Ani_D$, the $\infty$-category of animated $D$-algebras. For any polynomial $D$-algebra $P$, we can consider the quotients $P/tP$ and $P/t^pP$ which are both $\F_p$-algebras. We can animate the functor (see \cite[Section 5.1.4]{cesnavicius2024purity}) to get the base change functors
 \[(-/^L t),(-/^Lt^p)\colon \Ani_D\to \Ani_{\F_p}.\]
 Now, for each polynomial $D$-algebra $P$, we consider the $\F_p$-algebra map $(-)^p\colon P/tP\to P/t^pP$. Animating again, this defines a natural transformation $(-)^p\colon (-/^L t)\to (-/^Lt^p)$ as desired. 

The fact that this gives the ordinary $p$-power map on associated discrete quotients follows because when Kan extending the above functor defined on polynomial algebras to discrete algebras $\Alg_D\to \Alg_{\F_p}$ this retrieves the ordinary quotients and the usual $p$-power map. 
 
 To see the effect on $\pi_1$, consider $E=D[X]/(tX)$. An element $x\in R[\varpi]$ corresponds to a $D$-algebra map $E\to R$ sending $X\mapsto x$. By functoriality, it suffices to prove that the induced map $\pi_1(E/^Lt)=E[t]\to E[t^p]=\pi_1(E/^Lt^p)$ sends $X$ to $X^p$. Note that $E$ is 
%  we have $E=D[X]\otimes^L_{D[Z],0,tX}D$, 
 presented by the Koszul complex $D[X]\xrightarrow{tX}D[X]$. Base changing $E$ to $(-/^L t)$ and $(-/^Lt^p)$ gives the quasi-ideal (see \cite{drinfeld2021notionringgroupoid}) description of the corresponding animated algebra map $(-)^p$, where $\widetilde u$ is a lift of $u$ to $(D/t^pD)[X]$.
\[\begin{tikzcd}
{D/tD[X]} \arrow[d, "u\mapsto(tX)^{p-1}\widetilde u^{\,p}"] \arrow[r, "0"] & {D/tD[X]} \arrow[d, "(-)^p"] \\
{D/t^pD[X]} \arrow[r, "tX"]                         & {D/t^pD[X]}                 
\end{tikzcd}\]
The identification of the kernels of the horizontal complexes with $E[t]$ and $E[t^p]$ is given by 
\[D/tD[X]\xrightarrow{\cdot X} E[t]\]
 and
 \[t^{p-1}(D/t^pD[X])\xrightarrow{\cdot\frac{X}{t^{p-1}}}E[t^p].\]
Following the isomorphisms, we see that $X\mapsto X^p$ as desired.

\end{rem}

This paper deals with the following topic.

\begin{definition}\label{definition:perfectoid_ideal}Let $R$ be a perfectoid ring. An ideal $J\subseteq R$ is called \emph{perfectoid} if $R/J$ is a perfectoid ring.
\end{definition}

\begin{rem}\label{remark:perfectoidization_semipfd}Thanks to Bhatt's surjectivity theorem \cite[Proposition 7.11]{bhatt2022Prismsa}, we can see that for any ideal $I\subseteq R$, the smallest perfectoid ideal containing $I$, denoted by $I_{\pfd}$, has the property that $R\to R/I_{\pfd}$ is initial among maps to perfectoid rings that send $I$ to zero. See \Cref{lemma:initexist} and \Cref{remark:perfectoidization_of_ideals} for details on this reasoning. A great deal of this article is devoted to understanding the structure of $I_{\pfd}$ for any ideal $I\subseteq R$.
\end{rem}
 
\section{Perfectoid ideals modulo \texorpdfstring{$p$}{p}}\label{section:pfdideals_modp}

In this section, the goal is to prove \Cref{theorem:eid-tilt-ideals}. Throughout this section, $R$ is a fixed perfectoid ring and $\varpi$ is as in \Cref{notation:varpi}.

\subsection{Tilting equivalence}

In this section, we recall a form of the tilting equivalence in \Cref{proposition:eidtilt}. See \cite[Proposition 2.1.9]{cesnavicius2024purity} for the usual tilting equivalence and \cite[Proposition 3.2.6]{bhatt2025Aspects} for a general statement. The difference between \Cref{proposition:eidtilt} and \cite[Proposition 3.2.6]{bhatt2025Aspects} is that we are taking the underived quotient by $\varpi$, which is different in the presence of torsion. We give a proof with no mention of deformation theory. The two keys are the fact that the tilt is actually determined by the classical quotient by $\varpi^p$ and \Cref{lemma:principal_miracle}.

For the sake of clarity in this article, we give a name to the target category. This category captures perfectoid rings "purely algebraically" in the sense that rings are not required to be complete with respect to any topology in this category.

\begin{definition}\label{definition:eid} A ring $S$ is \emph{eidetic\footnote{ The name is loosely inspired by \href{https://en.wikipedia.org/wiki/Eidetic_reduction}{eidetic reduction}.} perfectoid} if $pS=0$ and there exists an $\alpha\in S$ such that $\alpha^p=0$ and
\[(-)^p\colon S/\alpha S\to S\]
is an isomorphism.
The \emph{category of eidetic perfectoid rings} has eidetic perfectoid rings as objects; its morphisms are ring maps $S\to S'$ such that $S'$ is \emph{relatively perfect} as an $S$-algebra. We will denote it by $\Eid$. If $S$ is eidetic perfectoid, then we denote by $\Eid_S$ the coslice category, whose objects we call eidetic perfectoid $S$-algebras.
    
\end{definition}

\begin{example}\label{example:eid} Let $R$ be a perfectoid ring and $\varpi$ as in \Cref{notation:varpi} with $\varpi^p=pu$ for a unit $u\in R$. Then $S=R/pR$ is eidetic perfectoid, $\alpha$ being the image of $\varpi$ in the quotient.  We will see that any eidetic perfectoid ring is of this form in \Cref{proposition:eidtilt}. As a trivial example, if $R$ is perfect in characteristic $p$, then taking $\alpha=0$ works; the category of perfect $\F_p$-algebras is equivalent to the category $\Eid_{\F_p}$.    
\end{example}

Note the following.

\begin{lemma}\label{lemma:relperf}Let $S$ be an eidetic perfectoid ring and $S\to S'$ be a map. The following are equivalent.
\begin{enumerate}
    \item The $S$-algebra $S'$ is relatively perfect, meaning that the relative Frobenius $F_{S'\mid S}$ is an isomorphism.
    \item As a ring, $S'$ is eidetic perfectoid and $\ker(F_{S'})=\alpha S'$ for any $\alpha\in S$ with $\ker(F_S)=\alpha S$.
\end{enumerate}
    
\end{lemma}
\begin{proof}Recall that the relative Frobenius is defined as the unique map from the ring pushout
\[\begin{tikzcd}
S \arrow[d, "(-)^p"'] \arrow[r]     & S' \arrow[d] \arrow[rd, "(-)^p", bend left]              &    \\
S \arrow[r] \arrow[rr, bend right] & {S'\otimes_{S,(-)^p}S} \arrow[r, "F_{S'\mid S}", dashed] & S'
\end{tikzcd}\]
As $S$ is eidetic perfectoid, we can take some $\alpha\in S$ with $(-)^p\colon S/\alpha S\to S$ being an isomorphism and we can understand the $S$-algebra $S\xrightarrow{(-)^p}S$ as the quotient map $S\to S/\alpha S$ with inverse given by $s_1\otimes s_2\to s_1s_2^{1/p}$ -- the $1/p$-root is well-defined because the quotient is modulo the kernel of the Frobenius. Therefore we can rewrite the relative Frobenius $F_{S'\mid S}$ as the $p$-power map
\[(-)^p\colon S'/\alpha S'\to S'.\]
Now the equivalence in the statement is clear.
\end{proof}

% \begin{example}\label{example:eid} Let $R$ be a perfectoid ring. Let $\varpi\in R$ be an element such that $\varpi^p=pu$ where $u\in R^{\times}$. Then $S=R/pR$ is eidetic perfectoid, $\alpha$ being the image of $\varpi$ in the quotient.  We will see that any eidetic perfectoid ring is of this form in \Cref{proposition:eidtilt}. As a trivial example if $R$ is perfect in characteristic $p$, then taking $\alpha=0$ works.
    
% \end{example}

We prove the following version of the tilting equivalence, which is a small addendum to the usual tilting equivalence.

\begin{proposition}
    \label{proposition:eidtilt} Let $R$ be a perfectoid ring. Let $S=R/\varpi^p R$. Then $S$ is eidetic perfectoid and we have an equivalence of categories 
\[\Pfd_R\xrightarrow{\sim}\Eid_S\]
given by $R'\mapsto R'/\varpi^p R'$.

% with inverse given by sending $S'\in \Eid_S$ to the $R$-algebra
% $$R=W(S^{\flat})/\ker(\theta_R)\to W(S'^{\flat})/\ker(\theta_R)W(S'^{\flat}).$$
    
\end{proposition}

\begin{proof} 
Using the tilting equivalence \cite[Proposition 2.1.9]{cesnavicius2024purity} and that we have a functorial isomorphism $(-)^{\sharp}\colon R'^{\flat}/(\varpi^{\flat})^pR'^{\flat}\to R'/\varpi^pR'$ in $\varpi$-complete perfectoid $R$-algebras (given by \Cref{lemma:principal_miracle}, for example), we can suppose that $R$ is in characteristic $p$ and show that the reduction modulo $\varpi^p$
\[\Perf_R^{\varpi}\xrightarrow{\mod \varpi^p} \Eid_S\]
where $\Perf_R^{\varpi}$ denotes the category of $\varpi$-complete perfect $R$-algebras, is an equivalence of categories. Note that one can see that reduction modulo $\varpi^p$ lands in relatively perfect $S$-algebras using the criterion of \Cref{lemma:relperf}. We claim that an inverse functor is given by the inverse limit perfection functor. Note that it lands in complete objects, for example by \cite[Lemma 3.3]{ishizuka2026decompositiontheoremperfectoidrings}. For any $R'\in \Perf_R^{\varpi}$ the natural transformation
\[R'\to (R'/\varpi^p R')^{\flat}\]
is an isomorphism because $R'$ is assumed to be $\varpi$-complete and perfect: indeed, we can postcompose the above with the isomorphism $((-)^{p^n})_{n\geq 0}\colon \varprojlim_{(-)^p}R'/\varpi^p R'\to \varprojlim_{n\geq 0} R'/\varpi^{p^{n+1}}$ and then the composition of these maps is naturally identified with $\varpi$-completion. Using this map when $R'=R$ and \mbox{$(-)^{\flat}$-functoriality} to define the $R$-algebra structures, we therefore see that $(-)^{\flat}\colon \Eid_S\to \Perf^{\varpi}_R$ is a well-defined left inverse.
Now, if $S'\in \Eid_S$, we only need to show that the kernel of the natural transformation
\[S'^{\flat}\to S'\]
is given by $\varpi^p S'^{\flat}$. This now follows from \Cref{lemma:principal_miracle}, concluding the proof.

\end{proof}

As mentioned above, taking the derived quotient by $\varpi^p$ is not needed in \Cref{proposition:eidtilt}. We take the opportunity to mention the following instance of a similar phenomenon.

\begin{lemma}\label{lemma:pComplete_underived}Let $R\to R'$ be a morphism of perfectoid rings. Suppose that the map $R/\varpi^p R\to R'/\varpi^p R'$ is flat. Then $R\to R'$ is a $\varpi$-completely flat map. 
\end{lemma}
\begin{proof}
Here the subtlety is that $R$ and $R'$ are not necessarily $\varpi$-torsion-free -- $\varpi$-complete flatness is equivalent to the map of animated rings $R/^L\varpi^p\to R'/^L \varpi^p$ being flat; see \cite[Section 5.2.1]{cesnavicius2024purity}. Here we only suppose that the usual quotient $R/\varpi^p R\to R'/\varpi^p R'$ is flat.\footnote{Note that for maps between arbitrary $p$-complete rings, an analogous statement is false, as shown by $\Z_p\to \Z_p\abracket{x}/(px)$.} Therefore we need to show that under this hypothesis, the map $R[\varpi^p]\otimes_{R/\varpi^p R} R'/\varpi^p R'\to R'[\varpi^p]$ is automatically an isomorphism. To this end, using that the above can be rewritten as $R^{\flat }[\varpi^{\flat p}]\otimes_{R^{\flat}/\varpi^{\flat p}R^{\flat}}R'^{\flat}/\varpi^{\flat p}R'^{\flat}\to R'^{\flat }[\varpi^{\flat p}]$, it suffices to show the claim below.

\begin{claim}Let $A\to A'$ be a map of perfect rings and $a\in A$ such that $A/aA\to A'/aA'$ is flat. Then
\[A[a]\otimes_{A/aA}A'/aA'\to A'[a]\]
is an isomorphism.
\end{claim}
We write $B=A/aA$ and $B'=A'/aA'$. Note that as $A[a]=A[a^2]$ and $A'[a]=A'[a^2]$ because those rings are perfect, the natural maps $A[a]\to B$ and $A'[a]\to B'$ are injective. Denote by $J$ and $J'$ the respective images. The statement is now equivalent to $B'J=J'$. Let $j\in J'$; by definition there is some $x\in A'$ with $xa=0$ which goes to $j$ in $B'$. Taking the Frobenius inverse of this equation and setting $y=x^{1/p}$, we get $ya^{1/p}=0$ in $A'$. Note that therefore $\overline{y}$, the image of $y$ in $B'$, is in $\Ann_{B'}(a^{1/p})$. Flatness of $B\to B'$ gives that $\overline{y}=\sum \alpha_i t_i$ with $\alpha_i\in B'$ and $t_i\in \Ann_{B}(a^{1/p})$. Now
$j=\sum \alpha_i^p t_i^p$. But now note that $t_i^p\in J$: indeed, as $t_ia^{1/p}=0$ in $B$, we have $\widetilde{t_i}^pa=a^ps_i$ in $A$ for some $s_i\in A$ and a lift $\widetilde{t_i}$ of $t_i$. Therefore $a(\widetilde{t_i}^p-a^{p-1}s_i)=0$, implying that $t_i^p\in J$, which concludes the proof.

% To prove the claim suffices to show that $B/J\otimes_B B'\to B/J'$ is an isomorphism. Surjectivity is evident. 

% Let's write $R=A\times_C B$ and $R'=A'\times_{C'}B'$ with $A=R/R[p]$, $B=(R/p)_{\red}$ and $C=(A/pA)_{\red}$ and similarly for $R'$ with $(-)'$ notation -- see \cite[Section 2.1.3]{cesnavicius2024purity} for this decomposition. Note that $R/pR=A/pA\times_C B$ and $S/pS=A'/pA'\times_{C'}B'$. Write $I=\ker(B\to C)$ and $I'=\ker(B'\to C')$. Those ideals canonically identify to $R[p]$ and $R'[p]$ respectively. Now in the below diagram depicting exact sequences,
% \[\begin{tikzcd}
% 	0 & I & B & C & 0 \\
% 	0 & {I'} & {B'} & {C'} & 0
% 	\arrow[from=1-1, to=1-2]
% 	\arrow[from=1-2, to=1-3]
% 	\arrow[from=1-3, to=1-4]
% 	\arrow["{(-)\otimes_{R/pR}R'/pR'}", shift right=4, dashed, from=1-3, to=2-3]
% 	\arrow[from=1-4, to=1-5]
% 	\arrow[from=2-1, to=2-2]
% 	\arrow[from=2-2, to=2-3]
% 	\arrow[from=2-3, to=2-4]
% 	\arrow[from=2-4, to=2-5]
% \end{tikzcd}\]
% we see that the second line is obtained by the flat tensor $(-)\otimes_{R/pR}R'/pR'$ of the first line, in particular $I\otimes_{R/pR}R'/pR'=I'$ as desired.
\end{proof}

We give a precise example of an eidetic perfectoid ring.

\begin{example}\label{example:eid-detailled}Consider the perfectoid ring $\mathbb{Z}_p^{\mathrm{cycl}}$. It can be
described as the $p$-adic completion of
\[
\frac{\mathbb{Z}_p[z_n\mid n\geq 1]}
{\left(
z_1^{p-1}+z_1^{p-2}+\cdots+z_1+1,\;
z_{n+1}^p-z_n\mid n\geq 1
\right)}
\]
where $z_n$ is sent to a primitive $p^n$-root of unity. We can consider, up to a unit, that $\varpi^p$ is $\xi_p-1$ where $\xi_p$ is a primitive $p$-th root of unity. Then the corresponding eidetic perfectoid ring can be described as (with $x=z-1$)
\[S=\F_p[z^{1/p^{\infty}}]/(z-1)=\F_p[x^{1/p^{\infty}}]/(x).\]
An interesting phenomenon that the author wants to point out is that the semiperfectoid ring $\Z_p^{\cycl}\abracket{t^{1/p^{\infty}}}/(t-1)$ modulo $\xi_p-1$ is (with $y=t-1$)
\[S\otimes_{\F_p}S= \F_p[x^{1/p^{\infty}},y^{1/p^{\infty}}]/(x,y).\]
This ring is not eidetic (over $S$) as $(x^{1/p},y^{1/p})$ is the kernel of the Frobenius, and this ideal is not principal. Indeed, a generator would lift at some finite level $\F_p[x^{1/p^n},y^{1/p^n}]/(x,y)$ where it is clear that it is absurd by Krull's height theorem in the ambient polynomial ring, for example. Note that this argument also shows that $\Z_p^{\cycl}\widehat{\otimes}_{\Z_p}\Z_p^{\cycl}$ is not perfectoid.
% An intersting phenomenon that the author wants to point out is that $\Z_p^{\cycl}\widehat{\otimes}_{\Z_p^{\cycl}}\Z_p^{\cycl}$ has the same eidetic reduction as $\Z_p^{\cycl}\abracket{t^{1/p^{\infty}}}/(t-1)$ by the above calculation.

The above will be studied in more detail in \Cref{section:program}.
\end{example}

\begin{rem}\label{remark:lost_info_semipfd}Note that the tilting correspondence \Cref{proposition:eidtilt} breaks down for semiperfectoid rings and loses information with respect to the perfectoidization. Take $R$ to be a $p$-torsion-free perfectoid ring. Then $R\abracket{t^{1/p^{\infty}}}/(t)$ and $R\abracket{t^{1/p^{\infty}}}/(t-p)$ have the same reduction modulo $p$ (and therefore also the same tilt). But these rings do not have the same perfectoidization: namely, after a $p$-completely faithfully flat base change to a perfectoid ring where $p$ satisfies the condition on $r$ in \cite[Notation 3.17]{ishizuka2026perfectoidizaitonfinitealgebrasperfectoid} (this base change can be taken to be the identity map if there are already such elements in $R$), the perfectoidizations differ. Namely, one is given by $R$ and the other is given by \cite[Proposition 3.18]{ishizuka2026perfectoidizaitonfinitealgebrasperfectoid}.
    
\end{rem}

\subsection{Ideal correspondence}

The analogous definition of perfectoid ideals for eidetic perfectoid rings becomes a really simple algebraic condition. Recall that for a ring $S$ in characteristic $p$ which is semiperfect, if $I\subseteq S$ is an ideal, then
\[I^{[p]}=\{i^p\mid i\in I\}\]
is also an ideal, namely the Frobenius image of the surjective Frobenius map.
\begin{definition}\label{definition:eid-ideals} Let $S$ be an eidetic perfectoid ring. An ideal $I\subseteq  S$ is called \emph{eidetic perfectoid} if $I^{[p]}=I$. This is equivalent to the map $F\colon I\to I$ being surjective.
\end{definition}

We can now prove the main theorem of this section.

\begin{thm}\label{theorem:eid-tilt-ideals} Let $R$ be a perfectoid ring and $S=R/\varpi^p R$ the associated eidetic perfectoid ring (see \Cref{proposition:eidtilt}). Then reduction modulo $\varpi^p$ induces a bijection
\[\{J\subseteq  R \mid J \text{ is a perfectoid ideal}\}\to \{I\subseteq  S \mid I\text{ is an eidetic perfectoid ideal}\}\]
\end{thm}
\begin{proof}Note that perfectoid ideals (\Cref{definition:perfectoid_ideal}) correspond one-to-one to equivalence classes of surjections
$R\to R'$
to a perfectoid ring $R'$. Because those rings are $\varpi$-complete, surjections are preserved and can be checked modulo $\varpi^p$. Using the equivalence of categories \Cref{proposition:eidtilt}, we therefore see that those surjections are in one-to-one correspondence with surjections
\[S\to S'\]
with $S'\in \Eid_S$. Let $\alpha$ denote the image of $\varpi$ in $S$. For an ideal $I\subseteq  S$, the condition that $S/I$ is relatively perfect over $S$ is equivalent to the condition that
$(-)^p\colon S/(I,\alpha)\to S/I$
is an isomorphism (\Cref{lemma:relperf}). Note that $(-)^p\colon S/\alpha S\to S$ is an isomorphism by assumption and therefore there is an induced isomorphism $(-)^p\colon S/(I,\alpha)\to S/I^{[p]}$. Now, looking at the commutative diagram
\[\begin{tikzcd}
{S/(\alpha,I)} \arrow[r, "(-)^p"] \arrow[rd, "(-)^p"'] & S/I^{[p]} \arrow[d] \\
                                                       & S/I            
\end{tikzcd}\]
where the vertical arrow is the natural quotient map, we see using the two-out-of-three principle for isomorphisms that $S/I$ is relatively perfect over $S$ if and only if $I=I^{[p]}$. This concludes the proof.
\end{proof}

\begin{rem}\label{remark:ideal_correspondence_more_explicitly} The opposite direction in \Cref{theorem:eid-tilt-ideals} can be made more explicit. Note that the tilt $(-)^{\flat}\colon \Eid_S\to \Perf^{\varpi^{\flat}}_{R^{\flat}}$ preserves surjections (as a morphism between $\varpi^{\flat}$-complete objects is surjective if and only if it is modulo $\varpi^{\flat}$), so if $S\to S'$ is a surjection in $\Eid_S$ with kernel $I$, then the kernel of $S^{\flat}\to S'^{\flat}$ is given by $I^{\flat}=\{(i_n)_{n\in \N}\in S^{\flat}=\varprojlim_{(-)^p}S \mid i_n\in I\}$. Now, using the functor $W(-)/\ker(\theta)$, we see that the corresponding ideal in $R$ is $p$-adically generated\footnote{We mean by this the process of taking the usual generated ideal and then its $p$-adic closure.} by elements of the form
$$\lim_{n\to \infty} (\widetilde{i}_n)^{p^n}$$
where $\widetilde{i}_n$ is \emph{any} lift of $i_n$ to $R$ and the limit is taken in the $\varpi$-adic topology.       
See \Cref{lemma:approximations} for a subsequent remark on the above, in a formulation which is useful for the application in \Cref{section:program}. In particular, we explain how an element $i_n$ yields a more precise approximation of the above limit in the $\varpi$-adic topology as $n$ grows.
\end{rem}

\begin{lemma}\label{lemma:approximations}Let $R$ be a perfectoid ring and $r \in R$ an element. Consider the map $R\to R/(r)_{\pfd}$ and denote by $S\to S'$ the induced map at $\varpi^p$-reduction. Denote by $i=\overline{r}\in S$ the image in the quotient. Let $(i_n)_{n\in \N}$ be a compatible system of $p$-power roots of $i$ in $S$, such that $i_n$ is in the kernel of $S\to S'$ for every $n\in \N$. Let $\widetilde{i}_n$ be any lift of $i_n$ to $R$. Then
$$\alpha:=\lim_{n\to \infty} (\widetilde{i}_n)^{p^n}\in R$$
is in the kernel of $R\to R/(r)_{\pfd}$, admits a compatible sequence of $p$-power roots, and is equal to $r$ modulo $\varpi^p R$.

Moreover, for $n\in \N$, we have
$$(\widetilde{i}_n)^{p^n} \equiv \alpha \mod \varpi^{p(n+1)} R.$$
\end{lemma}
\begin{proof}
The fact that $\alpha$ is in the kernel of $R\to R/(r)_{\pfd}$ follows from inspecting the functoriality of $W(-)/\ker(\theta)$ (see, for example, \cite[Proposition 2.1.2]{FarguesFontaine2018} applied with $Q=X^p$). Now the last remark follows from the usual fact (which is also implicitly used in the existence of this limit) that if $x\equiv y \mod I$ for any ideal $I\subset R$ then $x^p\equiv y^p \mod pI+I^p$ by the binomial formula. Conduct an induction using this fact with $I=(\varpi^p)$ and use that $(p)\subseteq (\varpi^p)$ by assumption.
\end{proof}

\begin{rem}\label{remark:dine_isihizuka} Let $R$ be any perfectoid ring. Note that if $S=R/\varpi^p R$ is eidetic perfectoid, then eidetic perfectoid ideals are also in one-to-one correspondence with
\[\{K\subseteq R^{\flat} \mid K \emph{ is closed and radical}\}\leftrightarrow \{I\subseteq  S \mid I\emph{ is an eidetic perfectoid ideal}\}\]
where closed means closed in the $\varpi^{\flat}$-topology -- one gets the above claim using \Cref{theorem:eid-tilt-ideals} for the $\varpi^{\flat}$-complete perfectoid ring $R^{\flat}$. Now, combining the applications of \Cref{theorem:eid-tilt-ideals} to $R$ and $R^{\flat}$ and using the isomorphism $R^{\flat}/\varpi^{\flat} R^{\flat}\to R/\varpi R$, we recover \cite[Main Theorem A]{Dine2024} (with the same bijection as in their paper),
\[\{K\subseteq R^{\flat} \mid K \emph{ is closed and radical}\}\leftrightarrow \{J\subseteq  R \mid J \emph{ is a perfectoid ideal}\}\]
This slightly generalizes their statement to cases where the perfectoid ring $R$ is not necessarily $p$-torsion-free. This last correspondence is given by tilting and untilting exactly as described in \cite{Dine2024}. In particular, it preserves inclusion and intersection of ideals.
\end{rem}

Note that the following is a direct consequence of \Cref{theorem:eid-tilt-ideals}.
\begin{cor}\label{corollary:determined} Let $R$ be a perfectoid ring and let $J_1$ and $J_2$ be perfectoid ideals. Then $J_1=J_2$ if and only if $(J_1,\varpi^p)=(J_2,\varpi^p)$. In particular, as one can take a $\varpi$ such that $\varpi^p R=pR$, we have $J_1=J_2$ if and only if $(J_1,p)=(J_2,p)$.
    
\end{cor}

\section{On equations for perfectoidizations of ideals}\label{section:on_equations}

% \Cref{section:prism-env} and \Cref{section:Andre} provide recollections to give context for the article.

\begin{notation}\label{notation:PrismEnvelSection}Throughout \Cref{section:on_equations}, unless otherwise specified, $R$ is a perfectoid ring and we write $A=A_{\inf}(R)$. We also continue to follow \Cref{notation:varpi}.
    
\end{notation}

\subsection{Recollections on perfect prismatic envelopes}\label{section:prism-env}

In this section, we define perfect prismatic envelopes and give several key classical examples to illustrate them and explain their pertinence, which may be useful for future reference. We will use them in \Cref{section:equation} to show that perfectoid ideals are governed by universal equations.

\begin{rem}\label{remark:prismatic-enveloppe} 
% Let $R$ be a perfectoid ring and $A=A_{\inf}(R)$. 
Let $d\in A$ be a generator of the kernel of Fontaine's map $\theta_R\colon A\to R$. Let $(a_e)_{e\in E}\subseteq  A$ be a family of elements indexed by a set $E$. Then one can consider the following (derived) pushout in the category of (animated\footnote{One can ignore the derived pushout if one only cares about the perfection. But the correct object without perfection is the derived one; in this article, it is especially convenient for \Cref{example:key_flatness}, where the language of derived tensor products is more suitable. See \cite[Appendix A]{bhatt2022prismatizationpadicformalschemes} for an exposition of animated $\delta$-rings.}) $(p,d)$-complete $\delta$-rings:
\[\begin{tikzcd}
A\{E\} \arrow[r, "e\mapsto de-a_e"] \arrow[d, "e\mapsto 0"'] & A\{E\} \arrow[d] \\
A \arrow[r]                                                  & A\left\{\frac{E}{d}\right\}
\end{tikzcd}\]
% This $\delta$-$A$-algebra $A\left\{\frac{E}{d}\right\}$ has the following universal property: a map from this $\delta$-$A$-algebra is the same as map of $\delta$-rings $A\to B$ to a $\delta$-ring $B$ together with a choice of elements $b_e\in B$ such that $db_e=a_e$. 
Now consider
\[A\left\{\phi^{-\infty}\left(\frac{E}{d}\right)\right\}:=\left(\varinjlim_{n\in \N}A\left\{\frac{\phi^{-n}(E)}{\phi^{-n}(d)}\right\}\right)^{(p,d)\wedge}=\left(\varinjlim_{\phi}A\left\{\frac{E}{d}\right\}\right)^{(p,d)\wedge}\]
the $\phi$-perfection\footnote{Note that for any $n\in \N$, one sees that 
$A\to A\left\{\frac{\phi^{-n}(E)}{\phi^{-n}(d)}\right\}$
is naturally identified with
$A\xrightarrow{\phi^n}A\to A\left\{\frac{E}{d}\right\}$.} and then $(p,d)$-completion of the above. 
This is a perfect (therefore discrete \cite[Remark A.24]{bhatt2022prismatizationpadicformalschemes}) $(p,d)$-complete $\delta$-ring.
We call this the \emph{perfect prismatic envelope of $E$}. 
% As taking the $\phi$-perfection and $(p,d)$-completion is a left adjoint to the inclusion in the category of perfect and $(p,d)$-complete $\delta$-rings,
This object has the following universal property: a map
$$A\left\{\phi^{-\infty}\left(\frac{E}{d}\right)\right\}\to B$$
where $B$ is a perfect and $(p,d)$-complete $\delta$-ring, is the same as a ring map\footnote{Note that any ring map between perfect $\delta$-rings is automatically a $\delta$-ring map thanks to \cite[Corollary 2.31]{bhatt2022Prismsa}.} $A\to B$ together with a choice of elements $b_e\in B$ such that $db_e=a_e$.

Note also that as $A\left\{\phi^{-\infty}\left(\frac{E}{d}\right)\right\}$ is perfect, the element $d$ is a non-zero-divisor in this ring, being distinguished (\cite[Lemma 2.34]{bhatt2022Prismsa}). Therefore $\left(A\left\{\phi^{-\infty}\left(\frac{E}{d}\right)\right\},(d)\right)$ defines a perfect prism.
\end{rem}

For the sake of completeness and to provide context for this article, we reprove the following. The proof is a slightly more precise variation of \cite[Proposition 7.2, Corollary 7.3]{bhatt2022Prismsa}. Note that no transfinite induction is needed below, in contrast to the proof in \cite[Proposition 7.2]{bhatt2022Prismsa}.
\begin{lemma}\label{lemma:initexist}
    Let $I\subseteq  R$ be an ideal. Then there is a ring morphism 
    $$R/I\to(R/I)_{\pfd}$$
    to a perfectoid ring which is initial among maps to perfectoid rings.   
    
    Namely, let $I':=\theta_R^{-1}(I)$. If we set (see \Cref{remark:prismatic-enveloppe})
    \[C:=A\left\{\phi^{-\infty}\left(\frac{I'}{d}\right)\right\}\]
     then $R/I\to C/dC$ is this initial map.
    \begin{proof}
    % Because perfectoid rings are $p$-complete, witout loss of generality, we can suppose that $I$ is $p$-adically closed.
    
    % We go equivalent to prove the analogue statements for $(p,d)$-complete perfect $\delta$-rings.
 
 First, note that by construction $C$ is a $(p,d)$-complete perfect $\delta$-ring, so that $C/dC$ is perfectoid. 
     % Recall also that any ring map between perfect $p$-complete, $p$-torsion free $\delta$-ring map is automatically a $\delta$-map.
     
     To check that $R/I\to C/dC  $ has the correct universal property, let $A_{\inf}(R)\to B$ be a map to a perfect $(p,d)$-complete $\delta$-ring such that every element of $I'$ is sent to zero after reduction modulo $d$. This means that the image of every element in $I'$ is in the ideal $dB$. Using $d$-torsion-freeness, we get that there is a unique element $\frac{i'}{d}$ in $B$ for every element $i'\in I'$. It now follows from the universal property of $C$, as recalled in \Cref{remark:prismatic-enveloppe}, that there is a unique map $C\to B$ making the required diagram commute, concluding the proof.
    \end{proof}
\end{lemma}

\begin{rem}\label{remark:initexist}In the proof of \Cref{lemma:initexist}, we could replace $I'$ by any set of elements generating $I'$. In particular, in the case of a principal ideal $I=(r)$ for $r\in R$, if $a\in A_{\inf}(R)$ is any lift of $r$, then
\[(R/rR)_{\pfd}
=\left(A\left\{\phi^{-\infty}\left(\frac{a}{d}\right)\right\}\right)/(d).\]
Let us note that $A\left\{\phi^{-\infty}\left(\frac{a}{d}\right)\right\}$ is the perfection and $(p,d)$-completion of the $\delta$-ring (with $\delta(t_i)=t_{i+1}$ for $i\geq 0$)
\[A[t_0,t_1,\dots,]/(\delta^i(dt_0-a)).\]
Again, we can ignore the derived aspect because we are taking the perfection of this $\delta$-ring.
\end{rem}

\begin{rem}\label{remark:crys_comp}
An important result\footnote{Note that this isomorphism of group schemes can actually be seen as the engine for Hodge-Tate comparison; see \cite[Construction 5.10]{bhatt2022prismatizationpadicformalschemes}.} in computing prismatic envelopes is \cite
[Lemma 2.36]{bhatt2022Prismsa}, which shows that the squares in 
$$\begin{tikzcd}
\G_a^{\sharp} \arrow[d] \arrow[r] & W\times_{\G_a}\G_a^{\sharp} \arrow[r, "\proj"] \arrow[d] & W \arrow[d, "F"] \\
1 \arrow[r, "0"']                 & W \arrow[r, "\cdot p=FV"']                               & W               
\end{tikzcd}$$ 
are pullback squares at the level of group schemes over $\Z_{(p)}$, where $W$ stands for the Witt scheme. Namely, \cite
[Lemma 2.36]{bhatt2022Prismsa} gives the square on the right-hand side -- the further pullback to the neutral element immediately gives the identification $\tau\colon W[F]\xrightarrow{\sim} \G_a^{\sharp}$ of the kernel of the Frobenius.\footnote{ For the record, we note that when choosing coordinates of $\G_a^{\sharp}$ as
$$\G_a^{\sharp}\cong \Spec\left(\Z_{(p)}\left[\frac{t^{p^n}}{p^{v_p(p^n!)}}\right]\right)_{n\in \N}$$ 
and using Joyal coordinates for Witt vectors, the isomorphism $\tau$ is given on points, for a $\Z_{(p)}$-algebra $R$, by
$$(\alpha_n)\in W[F](R)\to ((-1)^{n}\alpha_n)\in \G_a^{\sharp}(R)$$
We also note that the vertical map $W\times_{\G_a}\G_a^{\sharp}\to W$ is expressed on points as $(x,y)\mapsto V^{-1}(x-\tau(y))$.} 

On the algebra side, as the algebra of functions on the Witt scheme is the free $\delta$-ring on an element, the above implies that for any $\delta$-$\Z_{(p)}$-algebra $A$ and any element $a\in A$, we have
\[A\left\{\frac{\phi(a)}{p}\right\}\cong \Z_{(p)}\left[\frac{t^{p^n}}{p^{v_p(p^n!)}}\right]\widehat{\otimes}^L_{\Z_{(p)}[t]}A,\]
where the $\Z_{(p)}[t]$-algebra structure on $A$ is given by $t\mapsto a$. If $A$ is $p$-torsion-free and $a$ is a non-zero-divisor in $A/pA$, then by \cite[Lemma 2.38]{bhatt2022Prismsa} the above tensor product is classical and the resulting ring is $p$-torsion-free. For future reference, we note that in this case, using this isomorphism, we have
\begin{equation}\label{equation:envelope_mod_p}
    A\left\{\frac{\phi(a)}{p}\right\}/(p)\cong A/(p,a^p)\left[t_i\right]/(t_i^p)_{i\geq 1}
\end{equation}

% The above also allows for this powerful comparison about prismatic envelopes 

\end{rem}

\begin{example}\label{example:key_flatness}A key flatness property goes as follows. Let $A\to P$ be a map of perfect prisms which are both transversal ($A/dA$ and $P/dP$ are $p$-torsion-free) and let $f\in P$ be such that $f$ is a non-zero-divisor in $P/(p,d)$ and such that\footnote{Note that our assumptions ensure that those quotients are also derived. Here using the animated \emph{language} to work with the correct objects and properties is really useful.} $A/(p,d)\to P/(p,d,f)$ is (faithfully) flat. Then
\[A\to P\left\{\frac{f}{d}\right\}\]
 is $(p,d)$-completely (faithfully) flat. To see this, let $D=A\left\{\frac{d}{p}\right\}$ which is discrete and $p$-torsion-free by our assumptions and the recalled calculation at the end of \Cref{remark:crys_comp} (using $d=\phi(\phi^{-1}(d))$). The same goes for $P_D=P\widehat{\otimes}_A^L D$. Consider $P\left\{\frac{f}{d}\right\}\widehat{\otimes}_A^L D=P_D\{\frac{f}{p}\}$ which is also discrete and $p$-torsion-free by the same argument: $P_D$ is $p$-torsion-free and by \Cref{equation:envelope_mod_p} $P_D/pP_D=P/(p,d)[t_i]/(t_i^p)$ so $f$ is a non-zero-divisor in this ring by assumption. Now, using our flatness hypothesis together with \Cref{equation:envelope_mod_p} we see that $D\to P_D\left\{\frac{f}{p}\right\}$ is $p$-completely (faithfully) flat as modulo $p$ this reads
\[
A/(p, d)[t_i] / (t_i^p) \longrightarrow \Big( P/(p, d, f)[t_i] / (t_i^p) \Big) [y_i] / (y_i^p)
\]
 Note that there is a natural augmentation map $D\to A/(p,d)$ defined by sending the variables $t_i$ to zero. Now, benefiting from this factorization we see
 \[P\left\{\frac{f}{d}\right\}\otimes_A^L A/(p,d)=P_D\left\{\frac{f}{p}\right\}\otimes_{D}^L A/(p,d) \]
 and, as $D\to P_D\left\{\frac{f}{p}\right\}$ was just shown to be $p$-completely (faithfully) flat, this concludes the proof.
\end{example}

\begin{example}\label{example:crystalisation} In this example, as we only care about the perfection, we can use the classical counterpart of the pushout in \Cref{remark:prismatic-enveloppe} safely -- we do not change the notation but implicitly take $\pi_0$ of the potentially animated objects and use the universal properties in the ordinary category of $\delta$-rings.

We show what happens for perfect prismatic envelopes when $p\in I$. We start with the case $(p)=I$. First, we claim that in $A\{\frac{p}{d}\}^{\wedge (p,d)}$, the element $\frac{p}{d}$ is a unit. Namely, if $u=\frac{p}{d}$, we have $p=ud$, so applying $\delta$ we get
\[\delta(p)\equiv u^p\delta(d) \mod (p,d),\]
implying the claim using $(p,d)$-completeness. Note that the argument applies in a symmetric manner for $A\left\{\frac{d}{p}\right\}$. We therefore deduce
\[A\left\{\frac{p}{d}\right\}\cong A\left\{\frac{d}{p}\right\}\]
Now, using also that $A$ is perfect, we can use \Cref{remark:crys_comp} to compute the reduction modulo $p$ of this ring, which is
\[(R^{\flat}/\varpi^{\flat p})[t_i]_{i\geq 1}/(t_i^p)\]
As a consequence, the perfection $A\left\{\phi^{-\infty} \left(\frac{d}{p}\right)\right\}$ is isomorphic to the Witt vectors of the perfection of $(R^{\flat}/\varpi^{\flat p})[t_i]_{i\geq 1}/(t_i^p)$, which is $(R/\varpi)_{\red}$.

Now, for the general case, note that if $p\in I$, then

$$A\left\{\phi^{-\infty}\left(\frac{I}{d}\right)\right\}\cong\left(A\left\{\phi^{-\infty}\left(\frac{p}{d}\right)\right\}\right)\left\{\phi^{-\infty}\left(\frac{I}{d}\right)\right\}.$$
Therefore we can assume that $p=d$ and apply the exact same argument as above to deduce that 
\[A\left\{\phi^{-\infty}\left(\frac{I}{d}\right)\right\}^{\wedge (p,d)}\cong W\left(\left(R^{\flat}/(\varpi,I)\right)_{\red}\right).\]
\end{example}

\begin{example}\label{example:TeichmullerEnvellope}Another simple and key example is when $a\in A$ is such that $\delta(a)=0$. Note the following.
\begin{claim}Let $A$ be a perfect $\delta$-ring which is $(p,d)$-complete where $d\in A$ is a distinguished element. Let $a\in A$ with $\delta(a)=0$. Then $a=0$ if and only if $d$ divides $a$.
\begin{proof}Let $R=A/d$ be the associated perfectoid ring and $R^{\flat}=A/p$ its tilt. If $a$ is divisible by $d$, then $\theta(a)=0$. Recall that there is $(a_n)\in R^{\flat}=\varprojlim_{(-)^p}R$ such that $[(a_n)]=a$. We deduce that $a_n\in R$ is nilpotent for every $n\in \N$, and therefore zero, as perfectoid rings are reduced. As a consequence, $a=0$.
\end{proof}
\end{claim}
Now we can deduce that 
\[A\left\{\phi^{-\infty}\left(\frac{a}{d}\right)\right\}=A/(a^{1/p^{\infty}})^{\wedge (p,d)}.\]
Indeed, we see that $a^{1/p^n}=0$ in this ring by the above claim for every $n\in \N$. Moreover, this ring is $d$-torsion-free, so as $d\frac{a}{d}=a=0$, we deduce that $\frac{a}{d}=0$ in this ring, and the same holds for $\phi^{-n}\left(\frac{a}{d}\right)$ for every $n\in \N$. Now, we see that there is a well-defined map
\[A/(a^{1/p^{\infty}})^{\wedge (p,d)}\to A\left\{\phi^{-\infty}\left(\frac{a}{d}\right)\right\}^{(p,d)\wedge }\]
which is surjective. Also, note that the map given by sending $\frac{a}{d}\mapsto 0$ in $A/(a^{1/p^{\infty}})^{\wedge (p,d)}$ is also surjective and realizes the above map as a section, concluding the proof.
\end{example}

\begin{rem}\label{remark:perfectoidization_of_ideals}Together with André's lemma \cite[Theorem 7.14]{bhatt2022Prismsa} (see \Cref{lemma:Andre_analysis} and \Cref{cor:Andre_cover} for a recollection using perfected prismatic envelopes), the calculation in \Cref{example:TeichmullerEnvellope} shows Bhatt's surjectivity theorem \cite[Proposition 7.11]{bhatt2022Prismsa}: the map $R\to (R/I)_{\pfd}$ is surjective, and the kernel of this map is the smallest perfectoid ideal containing $I$.
\end{rem}

\subsection{Perfect prismatic envelopes and André's lemma}\label{section:Andre}

Having presented prismatic envelopes together with some examples, we take the opportunity to recall a proof of André's lemma in this language. The goal of this short section is to show the close relationship between typical André covers and perfectoid ideals, see \Cref{proposition:idealpfd_and_andre}.

\begin{lemma}[{André's lemma \cite[Theorem 7.14]{bhatt2022Prismsa}}]\label{lemma:Andre_analysis}Let $A$ be a perfect prism, let $d\in A$ be a distinguished element, and let $m(t)\in A[t]$ be a monic polynomial.
% Let $m(t)\in A[t]$ be a monic polynomial Then the map of prisms
% \[A\to A\abracket{y^{1/p^{\infty}}}\left\{\frac{m(y)}{d}\right\}^{(p,d)\wedge}\]
% is $(p,d)$-completely free and $(p,d)$-completely ind-syntomic cover. As a consequence (see the proof) 
Then the map to the perfect prismatic envelope
\[A\to A\abracket{y^{1/p^{\infty}}}\left\{\phi^{-\infty}\left(\frac{m(y)}{d}\right)\right\}\]
is $(p,d)$-completely faithfully flat.
\end{lemma}
\begin{proof}
By stability under derived $(p,d)$-complete base change\footnote{We mean that if $A'\to B$ and $A'\to C$ are maps between perfect $(p,d)$-complete $\delta$-$A$-algebras then the derived $(p,d)$-complete base change is discrete and coincides with the classical $(p,d)$-completion of the tensor product. One way to see this is that this derived base change is $W((B/p)\widehat{\otimes}^{\varpi^{\flat}}_{A'/p}(C/p))$. Indeed, this can be checked modulo $p$ using derived Nakayama (note that these rings are $p$-torsion-free). Modulo $p$, the derived $\varpi$-complete tensor product is the same as the classical completion of the usual tensor product because these rings are perfect.} we can assume that $A/d$ is $p$-torsion-free by \Cref{remark:universal_oriented_perfect_prism}. 
We first consider the map $A\to A\abracket{y^{1/p^{\infty}}}\left\{\frac{m(y)}{(d)}\right\}$. 
Note that this map is $(p,d)$-completely faithfully flat as it fits in \Cref{example:key_flatness}. As $A$ is perfect, we deduce that $A\xrightarrow{\phi^n}A \to A\abracket{y^{1/p^{\infty}}}\left\{\frac{m(y)}{(d)}\right\}$ is $(p,\phi^{-n}(d))$-completely faithfully flat. By \cite[\href{https://stacks.math.columbia.edu/tag/0H75}{Tag 0H75}]{stacks-project}, this turns out to be equivalent to $(p,d)$-complete faithful flatness, as $(p,d)=(p,\phi^{-n}(d)^{p^n})$.
Therefore, when taking the colimit over $n\in \N$ in the category of $(p,d)$-complete $\delta$-rings, the map
\[A\to A\abracket{y^{1/p^{\infty}}}\left\{\phi^{-\infty}\left(\frac{m(y)}{d}\right)\right\}\]
is a $(p,d)$-completely faithfully flat cover, as desired.
% $(p,d)$. Those rings are $d$-torsion free because they are perfect $\delta$-rings. Therefore taking the derived quotient modulo $d$ is the same as taking the usual quotient modulo $d$. At this stage we have a map of perfectoid rings which is flat after classical reduction modulo $p$, so by \Cref{lemma:pComplete_underived}, the map is actually derived $p$-completely flat, concluding.
\end{proof}

The following is therefore obtained by taking \Cref{lemma:Andre_analysis} modulo $d$.
\begin{cor}\label{cor:Andre_cover}Let $R$ be a perfectoid ring and let $m(t)\in R[t]$ be a monic polynomial. Then there exists a $p$-completely faithfully flat cover $R\to R'$ where $m(t)$ has a root and $R'$ is a perfectoid ring.
\end{cor}

To conclude this subsection on recollections about André's lemma and its consequences, we note the close relation between perfectoid ideals and typical André covers.

\begin{proposition}\label{proposition:idealpfd_and_andre}Let $R$ be a perfectoid ring and let $r\in R$ be an element. Consider the typical André cover $R\to R\abracket{y^{1/p^{\infty}}}/(y-r)$ from \Cref{cor:Andre_cover}. Then the constant terms of the elements in $(y-r)_{\pfd}\subseteq R\abracket{y^{1/p^{\infty}}}$ are exactly the elements in $(r)_{\pfd}\subseteq R$.

\end{proposition}
\begin{proof}
In the following diagram,
\[\begin{tikzcd}
R \arrow[r] \arrow[d]                                  & R/(r)_{\pfd} \arrow[d, "="] \\
R'=R\abracket{y^{1/p^{\infty}}}/(y-r)_{\pfd} \arrow[r] & R'/(y)_{\pfd}              
\end{tikzcd}\]
note that $(y)_{\pfd}\subset R'$ is given by the $p$-adic closure of $(y^{1/p^{\infty}})$, see \Cref{example:TeichmullerEnvellope}, for example. Now, setting $y^{1/p^n}$ to zero for all $n\in \N$ gives a surjective evaluation map
\[R'=R\abracket{y^{1/p^{\infty}}}/(y-r)_{\pfd} \to R/(r)_{\pfd}.\]
Therefore, it follows that the constant terms of the elements in $(y-r)_{\pfd}$ are exactly the elements in $(r)_{\pfd}$.
\end{proof}

\subsection{Equations for perfectoidizations of ideals}\label{section:equation}

We now pass to the main part of this section. 

One of the main questions in the context of perfectoidization of ideals is as follows: given $r\in R$, where $R$ is a perfectoid ring, how can we get to $(r)_{\pfd}$ from $r$? We show below that there exist (non-canonical) universal $p$-adic power series in infinitely many variables that solve this exact problem in a universal way.

\begin{rem}\label{remark:free_perf_delta_ring} Consider the free $\delta$-$A$-algebra $A\{x\}=A[x_0,\dots,x_n,\dots]$ with $\delta(x_n)=x_{n+1}$ for 
$n\geq 0$. We can consider its perfection and then its $(p,d)$-completion. We denote the resulting ring by
\[A\{\phi^{-\infty}(x)\}.\]
Note that in the category of perfect and $(p,d)$-complete $\delta$-$A$-algebras, a morphism out of this ring is the same as the choice of an element, namely the image of $x_0=x$.
Note that because this $\delta$-ring is perfect and $p$-complete, the natural map\footnote{The completion inside the Witt vectors is with respect to $t_0$.}
\[A\{\phi^{-\infty}(x)\}\to W(R^{\flat}\abracket{t_i^{1/p^{\infty}}}_{i\in \N})\]
sending $x=x_0$ to $(\overline{x},\overline{\delta(x)},\overline{\delta^2(x)},\dots)=(t_0,t_1,t_2,\dots)$ is an isomorphism.

Note that this ring modulo $d$ can therefore be identified with
\[R\abracket{t_i^{1/p^{\infty}}}_{i\in \N}\]
and that this ring (co)represents the functor $A_{\inf}(-)$ in the category of $p$-complete algebras over $R$.
% Now, one can consider the perfect prismatic envelope of $x$:
% \[\]
\end{rem}

\begin{rem}\label{remark:universal_oriented_perfect_prism}Consider the $\delta$-algebra underlying the universal oriented prism (\cite[Example 3.4]{bhatt2022Prismsa}), and denote by $d$ the specified distinguished element of this ring; then consider the $\phi$-perfection and $(p,d)$-completion of this $\delta$-algebra. We denote the resulting $(p,d)$-complete perfect $\delta$-ring by $A_0$.
A map from $A_0$ to another $(p,d)$-complete perfect $\delta$-ring $A$ is the same as the choice of a distinguished element $d\in A$. 

Using the same argument as in \Cref{remark:free_perf_delta_ring}, we see that the natural map
\[A_0\to W\left(\F_p\abracket{x_i^{1/p^{\infty}},x_1^{-1}}_{i\in \N}\right)\] 
is an isomorphism. Moreover, the latter can be identified (here the completion is with respect to $(p,x_0)$) with $\Z_p\abracket{x_i^{1/p^{\infty}},x_1^{-1}}_{i\in \N}$ as this is a $p$-torsion-free, $p$-complete residually perfect ring with the same reduction modulo $p$ as $A_0$. In the Witt vectors, the universal element $d$ corresponds to $(x_0,x_1,x_2,\dots)$, so in $\Z_p\abracket{x_i^{1/p^{\infty}},x_1^{-1}}_{i\in \N}$ it is identified with
\[\sum_{i=0}^{\infty}x_i^{1/p^i}p^i\]
using \cite[Example 9.3.49 (v) and Equation 9.3.47]{gabber2018foundationsringtheory}. Write $u=-\sum_{i=1}^{\infty}x_i^{1/p^i}p^{i-1}$ which is a unit in this ring. Now $d=x_0-pu$. As a consequence, we have an isomorphism
$$A_0/(d) \cong \Z_p\abracket{(pu)^{1/p^{\infty}},x_1^{\pm 1/p^{\infty}},x_2^{1/p^{\infty}},\dots}$$ 
and this is a $p$-torsion-free perfectoid ring which is weakly initial in the category of perfectoid rings, meaning that any perfectoid ring admits a map from it.

More precisely, a map from $A_0/(d)$ to a perfectoid ring $R$ is the same as choosing a unit $u\in R$ together with compatible $p$-power roots of $pu$, and also choosing an expression
\[u=\sum_{i=1}^{\infty} \sharp(v_i)^{1/p^i} p^{i-1}.\]
% We set $v_0=pu$.

\end{rem}

\begin{proposition}\label{prop:UniversalEquations}
    % Let $R$ be any perfectoid ring, choose an element $d\in A_{\inf}(R)$ that generates $\ker(\theta)$. Let $r\in R$ be any element and choose a lift $a\in A$.
There are power series (where the completion is with respect to $(p,x_0)$)
\[f_n\in \Z_p\abracket{x_i^{1/p^{\infty}},x_1^{-1},t_j^{1/p^{\infty}}}_{i,j\in \N}\]
indexed by $n\in \N$ such that for every perfectoid ring $R$, every element $r\in R$, any expression of $r$ as $r=\sum_{i=0}^{\infty}\sharp(r_i^{1/p^i})p^i$, and any choices as at the end of \Cref{remark:universal_oriented_perfect_prism}, we have
\[(r)_{\pfd}=\overline{(f_n(v_i,r_j)^{1/p^{\infty}})_{n\in \N}}\]
where the evaluations $f_n(v_i,r_j)\in R$ admit compatible sequences of $p$-power roots such that
\[r=\sum_{n=0}^{\infty}f_n(v_i,r_j)^{1/p^n}p^n.\]
    
\end{proposition}
\begin{proof}
    Consider the perfect $\delta$-ring $A_0$ from \Cref{remark:universal_oriented_perfect_prism}. Now consider the free perfect and $(p,d)$-complete $\delta$-$A_0$-algebra from \Cref{remark:free_perf_delta_ring}. We denote the resulting ring by $B$. Note that $B$ can be identified with
    \[\Z_p\abracket{x_i^{1/p^{\infty}},x_1^{-1},t_j^{1/p^{\infty}}}_{i,j\in \N}\]
    where the completion is with respect to $(p,x_0)$.
    Now consider the perfect prismatic envelope
    \[B\to B\left\{\phi^{-\infty}\left(\frac{x}{d}\right)\right\}\]
    which is surjective by Bhatt's surjectivity theorem \cite[Proposition 7.11]{bhatt2022Prismsa} (see also the recollection in \Cref{remark:perfectoidization_of_ideals}). Let $\sigma\in B$ be any preimage of $\frac{x}{d}$. Consider $\alpha=x-\sigma d$, which is in the kernel of the above map. By Witt functoriality, we see that all Witt components of $\alpha$ are in the kernel of the above map. Now one can consider the isomorphic $B$-algebra $B\left\{\phi^{-\infty}\left(\frac{\alpha}{d}\right)\right\}$; in this ring, every Witt component of $\alpha$ is zero. Denote these components by $f_n$. Reasoning as in \Cref{example:TeichmullerEnvellope}, we see that the kernel of this map is the $(p,d)$-closure of the ideal generated by $p$-power roots of Witt components of $\alpha$.
    In particular, reducing modulo $d$, we can deduce the claim of the proposition.
\end{proof}

\begin{rem}\label{remark:does_not_depend_from_choices} Note that it is clear from the proof that the power series $f_n$ are not unique or canonical, as they depend on the choice of a preimage $\sigma$ of $\frac{x}{d}$ in $B$. Note also that any such power series $f_n$ in \Cref{prop:UniversalEquations} do not depend on the choices made at the end of \Cref{remark:universal_oriented_perfect_prism}.
    
\end{rem}

\begin{rem}\label{remark:over_an_algebra} Note that in the category of perfectoid rings over a fixed base perfectoid ring $R$, one can make choices for roots of $(pu)$ as at the end of \Cref{remark:universal_oriented_perfect_prism} which remain valid for any perfectoid $R$-algebra. For example, say $R$ is the $p$-completion of $\Z_p[p^{1/p^{\infty}}]$. Then one can choose $u=1$, $v_1=1$, the obvious $p$-power roots of $p$, and $v_i=0$ for $i\geq 2$ as legitimate choices.

    In general, fixing a base perfectoid ring $R$, it follows from \Cref{prop:UniversalEquations} that there are power series (where the completion is $p$-adic)
\[g_n\in R\abracket{t_j^{1/p^{\infty}}}_{j\in \N},\]
indexed by $n\in \N$ such that for every perfectoid $R$-algebra $S$, every element $r\in S$ and any expression of $r$ as $r=\sum_{i=0}^{\infty}\sharp(r_i^{1/p^i})p^i$, we have $$(r)_{\pfd}=\overline{(g_n(r_j)^{1/p^{\infty}})_{n\in \N}}.$$
Also, the evaluations $g_n(r_j)\in S$ admit compatible sequences of $p$-power roots such that
$$r=\sum_{n=0}^{\infty}g_n(r_j)^{1/p^n}p^n.$$
\end{rem}

\begin{rem}\label{remark: pfd_approx_remark} Note that $\theta\colon A_{\inf}(R)\to R$ is surjective, so any element $r$ of a perfectoid ring can be written as a sum of the form
\[r=\sum_{n\in \N}\sharp(r_n)p^n\]
for a sequence $(r_n)_{n\in \N}$ in $R^{\flat}$. One of the key aspects of \Cref{prop:UniversalEquations} is that it provides \emph{preferred approximations} of the element $r$. In \cite[Remark 3.19]{ishizuka2026perfectoidizaitonfinitealgebrasperfectoid}, it is explained that not every element $\sharp(r_0)$ as above is in $(r)_{\pfd}$ and that \Cref{prop:UniversalEquations} provides only very special approximations.

% Note that understanding how to produce these special approximations of elements would solve the question of how to construct $I_{\pfd}$ from the ideal $I$.
    
\end{rem}

We conclude this section by pointing out that the existence of power series $f_n$ as in \Cref{prop:UniversalEquations} is equivalent to Bhatt's surjectivity theorem \cite[Proposition 7.11]{bhatt2022Prismsa}.

\begin{proposition}\label{proposition:equations_equivalent_to_Bhatt} Let $R$ be a perfectoid ring and $I\subseteq R$ be an ideal. The following statements are equivalent. 
    \begin{enumerate}
        \item The map $R\to (R/I)_{\pfd}$ is surjective. (Bhatt's surjectivity theorem \cite[Proposition 7.11]{bhatt2022Prismsa})
        \item For every $r\in R$ and every map $R\to S$ to a perfectoid ring $S$ such that $r$ goes to zero, there exists an expression $r=\sum_{i=0}^{\infty}\sharp(r_i^{1/p^i})p^i$ such that for every $i\in \N$, $\sharp(r_i)$ goes to zero in $S$.
    \end{enumerate}
    
\end{proposition}
\begin{proof}
    The implication from (1) to (2) follows from \Cref{prop:UniversalEquations}, where it is clear from the proof that we used Bhatt's surjectivity theorem. To prove that (2) implies (1), we can reduce to the case where $I$ is finitely generated and then, by induction, to the case where it is generated by one element. Then we can take $S=(R/(r))_{\pfd}$; now $\sharp(r_i)$ goes to zero in $S$ by assumption. Now notice that $R/\overline{(\sharp(r_i)^{1/p^{\infty}})_{i\in \N}}$ is a perfectoid $R$-algebra (see, for example, \cite[Proposition 2.1.11 (c)]{cesnavicius2024purity}) where $r=0$, so the map $R/\overline{(\sharp(r_i)^{1/p^{\infty}})_{i\in \N}}\to (R/(r))_{\pfd}$ admits a section, providing the desired surjectivity.
\end{proof}

% \subsection{Equations from perfected prismatic envelopes}\label{section:Ktheory}

The perspective of perfected prismatic envelopes allows us to give an alternative description of the perfectoidization of ideals. As we explain in \Cref{example:Teichmuller_case_Ktheory}, the equations obtained in \cite[Lemma 3.6]{antieau2024ktheorymathbfzpn} for prismatic envelopes highlight the role of completion in Bhatt's surjectivity theorem, even in the simplest case of elements admitting compatible sequences of $p$-power roots.
% \begin{rem}\label{remark:K-theoryequ}In this remark, we collect notations and results from \cite[Lemma 3.6]{antieau2024ktheorymathbfzpn}.
    
% \end{rem}

\begin{example}\label{example:Teichmuller_case_Ktheory}

% Let's fix $(A,d)$ perfect prism, with a distinguished element that we may take of the form $d=[\varpi^{\flat}]+pu$ and $r\in R=A/d$, $a\in A$ be any lift of $r$, write $s=a \mod p \in R^{\flat}$. 

Let $r\in R$ and let $a\in A$ be any lift of $r$. Using the tilting equivalence and \Cref{remark:initexist}, together with the notation from \Cref{remark:free_perf_delta_ring}, we see that
\[\left(R/(r)_{\pfd}\right)^{\flat}\cong \left(R^{\flat}[t_i^{1/p^{\infty}}]/(\delta^i(dx_0-a) \mod p)^{1/p^{\infty}}\right)^{\varpi^{\flat}\wedge}.\]

We can use equations from \cite[Lemma 3.6]{antieau2024ktheorymathbfzpn} to describe the right-hand side. Let's focus first on the case
\[\delta(a)=0\]
which is the case when $a$ is a Teichmüller lift of an element in $R^{\flat}$. The goal of this example is to highlight the role of completion in the computation and to show how to get directly to the expected form of the perfectoidization purely from the description of the prismatic envelope in \cite[Lemma 3.6]{antieau2024ktheorymathbfzpn}.

% We follow the notations from \cite[Lemma 3.6]{antieau2024ktheorymathbfzpn}. 
In this case (see \cite[Lemma 3.6, Remark 3.7]{antieau2024ktheorymathbfzpn}) we see that the $\delta$-ideal generated by $dx_0-a$ is given modulo $p$ by
\[s=\varpi^{\flat} t_0 \quad t_i^p=\varpi^{\flat p^{i+1}}\lambda_i t_{i+1}\]
where $\lambda_i$ is a unit in $R^{\flat}$ and $s=a \mod p$. We first show that, without completion, the following holds:
\[R^{\flat}[t_i^{1/p^{\infty}}]/(\delta^i(dx_0-a) \mod p)^{1/p^{\infty}}=\left(R^{\flat}\left[\frac{s^{p^n}}{\varpi^{\flat (n+1)p^{n}}}\right]\right)_{\perf}.\]
% I hope that this is can be explained by prismatic to crystalline comparison the right hand side looking like a (perfect) PD-hull of some sort. 
We prove by induction on $i\geq 0$ that
\[s^{p^i}=\varpi^{\flat (i+1)p^i}\sigma_i t_i\]
where $\sigma_i$ is a unit in $R^{\flat}$. Let's prove the induction step. If $s^{p^i}=\varpi^{\flat (i+1)p^i}\sigma_i t_i$, then
\[s^{p^{i+1}}=\varpi^{\flat (i+1)p^{i+1}}\sigma_i^p \underbrace{t_i^p}_{=\varpi^{\flat p^{i+1}}\lambda_i t_{i+1}}=\varpi^{\flat (i+2)p^{i+1}}\sigma_i^p\lambda_it_{i+1}\]
which proves our claim. 

Note, crucially, that therefore
\[s=\varpi^{\flat (i+1)}\sigma_i^{1/p^i}t_i^{1/p^i}.\]
In other words, $s$ is infinitely divisible by $\varpi^{\flat}$. This implies that $s$ is zero in the completion.

More precisely, for any $j\geq 0$, we can show by induction on $i\geq 0$ that
\[t_j^{p^i}=\varpi^{\flat i(p^{i+j})}u_{ij}t_{i+j}\]
where $u_{ij}\in R^{\flat \times}$. The base case is an equality. The induction step is similar to that of the previous induction. We have
\[t_j^{p^{i+1}}=\varpi^{\flat i(p^{i+j+1})}u_{ij}^pt_{i+j}^p=\varpi^{\flat (i+1)(p^{i+j+1})}u_{ij}^p\lambda_{i+j}t_{i+j+1}.\]
Now for every $j\geq 0$, we have
\[t_j^{p^{-j}}=\varpi^{\flat i}u_{ij}^{1/p^{i+j}}t_{i+j}^{1/p^{i+j}}.\]
Therefore $t_j^{p^{-j}}$ is infinitely $\varpi^{\flat}$-divisible. This implies that 
\[R^{\flat}\to \left(R^{\flat}\left[\frac{s^{p^n}}{\varpi^{\flat (n+1)p^n}}\right]\right)^{\varpi \wedge}_{\perf}\]
is surjective. Now, because we have shown that $t_i^{1/p^j}$ 
for $i\geq 0$ and $j\geq 0$ are in the kernel of the completion
\begin{align*}
    &\left(R^{\flat}[t_i^{1/p^{\infty}}]/(\delta^i(dt_0-a) \mod p)^{1/p^{\infty}}\right)^{\varpi^{\flat}\wedge}\\
    &=\left(R^{\flat}[t_i^{1/p^{\infty}}]/(s=\varpi t_0 \quad t_i^p=\varpi^{\flat p^{i+1}}\lambda_it_{i+1})^{1/p^{\infty}}\right)^{\varpi^{\flat}\wedge}
\end{align*}
% \[\left(R^{\flat}[t_i^{1/p^{\infty}}]/(\delta^i(dt_0-a) \mod p)^{1/p^{\infty}}\right)^{\varpi^{\flat}\wedge}=\left(R^{\flat}[t_i^{1/p^{\infty}}]/(s=\varpi t_0 \quad t_i^p=\varpi^{\flat p^{i+1}}\lambda_it_{i+1})^{1/p^{\infty}}\right)^{\varpi^{\flat}\wedge}\]
we see that this completion is the $\varpi^{\flat}$-completion of
\[R^{\flat}/(s)^{1/p^{\infty}},\]
which is the desired result. Note also that the above gives yet another independent proof of the form of the perfectoidization in the special case of elements admitting compatible sequences of $p$-power roots, as in \Cref{example:eid-detailled}.
\end{example}

% The following point out a general difficulty of computing perfecoidization of ideals. Namely, the perfectoidization of an ideal is always related to a subtle phenomenon related to ($p$-)completion.

% \begin{lemma}\label{lemma:derived_observation}Let $R$ be a $p$-torsion free perfectoid ring and $r\in R$ some element.
    
% \end{lemma}

\section{A computational approach to approximating a perfectoid ideal}\label{section:program}

In this section, we concentrate our attention on the following example
\[(t-1)_{\pfd}\subset \Z_p^{\cycl}\abracket{t^{1/p^{\infty}}}.\]

 We know from \cite[Proposition 3.18]{ishizuka2026perfectoidizaitonfinitealgebrasperfectoid} that the perfectoidization of the semiperfectoid ring $\Z_p^{\cycl}\abracket{t^{1/p^{\infty}}}/(t-1)$ is given by the evaluation map
\[\Z_p^{\cycl}\abracket{t^{1/p^{\infty}}}/(t-1)\to \Cacst{p}(\Z_p,\Z_p^{\cycl})\]
sending $t^{1/p^n}$ to the function
$$\Z_p\to \Z/p^n\Z\xrightarrow{i\mapsto \xi_{p^n}^i}\Z_p^{\cycl}.$$

We can look at the reduction modulo $\xi_p-1$ of this map, which reads, using \Cref{example:eid-detailled}, as the $\F_p[x^{1/p^{\infty}}]/(x)$-algebra morphism, 
\[E\colon \F_p[x^{1/p^{\infty}},y^{1/p^{\infty}}]/(x,y)\to \Cacst{p}(\Z_p,\F_p[x^{1/p^{\infty}}]/(x))\]
where now $y^{1/p^n}$ is sent to the function\footnote{To see the index shift, recall that $x=x^{1/p^0}$ is the class of $\xi_p-1$.}
$$\Z_p\to \Z/p^n\Z\xrightarrow{i\mapsto (1+x^{1/p^{n-1}})^i-1} \F_p[x^{1/p^{\infty}}]/(x).$$

By \Cref{remark:ideal_correspondence_more_explicitly}, we can search for elements in $\F_p[x^{1/p^{\infty}},y^{1/p^{\infty}}]/(x)$ whose images under the map $E$ are zero and which form a compatible system of $p$-power roots of $y$ in $\F_p[x^{1/p^{\infty}},y^{1/p^{\infty}}]/(x)$. Using \Cref{remark:ideal_correspondence_more_explicitly}, we see that this process produces $p$-adic approximations of an element in $(t-1)_{\pfd}$ which is congruent to $t-1$ modulo $\xi_p-1$.

\subsection{A truncated version of the problem}

We introduce some notation.
\begin{notation}\label{notation:setup_program}
We can truncate the problem at finite steps in order to make it amenable to computation on a computer. Namely, fix $M,L\in \N$. Set
\begin{itemize}
    \item $B_L=\mathbb{F}_p[a]/(a^{p^L})$
    \item $A _{L,M}= B_L[Y]/(Y^{p^M})$
\end{itemize}
where 
\begin{itemize}
    \item $a$ corresponds to $x^{1/p^L}\in \F_p[x^{1/p^{\infty}}]/(x)$,
    \item and $Y$ corresponds to $y^{1/p^M}\in \F_p[x^{1/p^{\infty}},y^{1/p^{\infty}}]/(x,y)$.
\end{itemize}
  Note that $B_L\subseteq \F_p[x^{1/p^{\infty}}]/(x)$ and $A_{L,M}\subseteq \F_p[x^{1/p^{\infty}},y^{1/p^{\infty}}]/(x,y)$, and that both these rings are filtered colimits of the $B_L$'s and $A_{L,M}$'s respectively. 

For $L\geq M-1$, we can consider the $B_L$-algebra map
\[E_{L,M}\colon A_{L,M}\to B_L^{p^M}\]
sending $Y\mapsto ((1+a^{p^{L-M+1}})^i-1)_{i=0,\dots,p^M-1}$. This is a truncated version of the map $E$ above.

Now, for fixed $L,M\in \N$, the map $E_{L,M}$ is an $\F_p$-linear map between finite-dimensional $\F_p$-vector spaces. The same goes for the Frobenius morphism on $A_{L,M}$.
\end{notation}  

\subsubsection{Finding approximations of an element congruent to $t-1$ in $(t-1)_{\pfd}$ that admit a compatible sequence of $p$-power roots.}\label{section:algori}

In the setup of \Cref{notation:setup_program}, we are now interested in a question of linear algebra over a finite field that we explain next. For some $n\in \N$, with $n\leq M,L$, which is fixed for this discussion, we are searching for elements $\alpha_n\in A_{L,M}$ with 
\begin{enumerate}
    \item $E_{L,M}(\alpha_n)=0$,
    % \item $\alpha_n^{p^n}=0$,
    \item $\alpha_n\equiv Y^{p^{M-n}} \mod a^{p^{L-n}}$
\end{enumerate}
In terms of the map (embedding the truncations in the context below)
\[E\colon \F_p[x^{1/p^{\infty}},y^{1/p^{\infty}}]/(x,y)\to \Cacst{p}(\Z_p,\F_p[x^{1/p^{\infty}}]/(x))\]
we are searching for elements in the kernel of this map (condition (1)) such that when lifting them to $\F_p[x^{1/p^{\infty}},y^{1/p^{\infty}}]/(x)$ they are congruent to $y^{1/p^n}$ modulo $x^{1/p^n}$ (condition (2)). Note that when lifting such an element in the unique class of $y$-degree $d$ satisfying $1/p^M\leq d< 1$ to $\F_p[x^{1/p^{\infty}},y^{1/p^{\infty}}]/(x)$, we get an element $f_n\in \F_p[x^{1/p^{\infty}},y^{1/p^{\infty}}]/(x)$ which is in the kernel of $E$ and such that
\begin{enumerate}
    \item $E(f_n)=0$,
    \item $f_n^{p^n}=y\in \F_p[x^{1/p^{\infty}},y^{1/p^{\infty}}]/(x)$.
\end{enumerate}
 
Using the notation of \Cref{lemma:approximations}, this is therefore an element that fits as an $i_n$ for $i_0=y$. Namely, let $i_m$ be determined by $f_n$ if $m\leq n$, and pick arbitrary $p$-power roots in the ideal to complete the sequence (to complete the argument and match the situation in \Cref{lemma:approximations}, we just care about the existence of the rest of the sequence).

Further translating in terms of the original map 
\[\Z_p^{\cycl}\abracket{t^{1/p^{\infty}}}/(t-1)\to \Cacst{p}(\Z_p,\Z_p^{\cycl}),\]
if we take any lift of the above $f_n$ in $\Z_p^{\cycl}\abracket{t^{1/p^{\infty}}}$, which we denote by $\widetilde{f}_n$, we get using \Cref{lemma:approximations} that
\[\widetilde{f}_n^{p^n}=\alpha \mod (\xi_p-1)^{n+1}\]
and $\alpha \in (t-1)_{\pfd}$ is an element with $\alpha=t-1 \mod (\xi_p-1)$, as in \Cref{lemma:approximations}.
Therefore, when computing this exponentiation, we can simplify the calculation modulo $(\xi_p-1)^{n+1}$.

% Therefore, we can write a SageMath program that searches for elements in the kernel of $E_{L,M}$ with the property that they are zero when raised to the $p^M$-th power.

% Such a SageMath program is available at the following \href{https://github.com/leo-nc-git/perfectoid_kernel}{GitHub repository}.

It turns out that such a problem has exponential growth and is very hard to solve in general. On the author's personal computer, using the SageMath code available at this \href{https://github.com/leo-nc-git/perfectoid_kernel}{GitHub repository}, the following results have been found.

\begin{proposition}\label{prop:kernel_results_2_2}For $p=2$, a solution $f_2\in \F_2[x^{1/2^{\infty}},y^{1/2^{\infty}}]/(x)$ exists at the truncated level $M=4$ and $L=3$. There are \emph{no solutions} at lower levels $M<4$ or $L<3$.
 One of the solutions is, for example,
\[
\begin{aligned}
f_2={}&x^{1/2}y^{1/16}+x^{5/8}y^{1/16}\\
&+x^{1/4}y^{1/8}+x^{3/8}y^{1/8}\\
&+x^{3/8}y^{3/16}+y^{1/4}.
\end{aligned}
\]
As a consequence, taking a lift $\widetilde{f}_2\in \Z_2^{\cycl}\abracket{t^{1/2^{\infty}}}$ (see \Cref{example:eid-detailled} for the notation) and computing its fourth power gives the following expression modulo $8$:
% This is the lift (t^{1/16}-1)^4+(zeta-1)^3(t^{1/16}-1)^3
% +((zeta-1)^2+(zeta-1)^3)(t^{1/16}-1)^2+((zeta-1)^4+(zeta-1)^5)(t^{1/16}-1).
\[
\begin{aligned}
\widetilde{f}_2^{\,2^2}\equiv{}&t-1\\
&+2t^{1/8}\Bigl(1-\xi_{16}^2-\xi_{16}^4-\xi_{16}^6
 +(1+\xi_{16}^2+\xi_{16}^4+\xi_{16}^6)t^{1/4}\\
&\qquad +(1+\xi_{16}^2-\xi_{16}^4+\xi_{16}^6)t^{1/2}
 +(1-\xi_{16}^2+\xi_{16}^4-\xi_{16}^6)t^{3/4}\Bigr)\\
&+4(1+t^{1/2})\Bigl(\xi_{16}^2+\xi_{16}^3+\xi_{16}^4+\xi_{16}^7
 +(\xi_{16}^2+\xi_{16}^4+\xi_{16}^5+\xi_{16}^7)t^{1/16}\\
&\qquad +(\xi_{16}+\xi_{16}^5)(t^{1/8}+t^{3/8})
 +(1+\xi_{16}^2+\xi_{16}^5+\xi_{16}^7)t^{3/16}\\
&\qquad +(\xi_{16}^3+\xi_{16}^4+\xi_{16}^6+\xi_{16}^7)t^{1/4}
 +(\xi_{16}+\xi_{16}^2+\xi_{16}^3+\xi_{16}^4)t^{5/16}\\
&\qquad +(1+\xi_{16}+\xi_{16}^2+\xi_{16}^3)t^{7/16}\Bigr)\pmod{8}.
\end{aligned}
\]
This element is congruent modulo $8$ to some $\alpha\in (t-1)_{\pfd}$ such that $t-1\equiv\alpha\mod 2$ and $\alpha$ admits a compatible sequence of $p$-power roots.

\end{proposition}

\begin{rem}\label{remark:github}Let us explain more about the \href{https://github.com/leo-nc-git/perfectoid_kernel}{SageMath script}. In this public repository, a script is provided that allows one to reproduce the calculations and results discussed in \Cref{prop:kernel_results_2_2}, for example. The author used this SageMath script to get the above results. Namely, it implements the algorithm described in \Cref{section:algori} to search for an $f_n$ and a $\widetilde{f}_n^{\,p^n}$, which is directly computed modulo $(\xi_p-1)^{n+1}$ to simplify the computations.

One can see by experimenting, as the problem is exponential, that it very rapidly becomes infeasible on an ordinary computer.
\end{rem}

\begin{rem}\label{remark:number_of_solutions}The number of solutions for an $f_2$ at level $M=4$ and $L=3$ is already huge. Namely, there are already $2^{80}$ solutions at this level. It is revealing of the computational complexity of the problem that at smaller levels, there are no solutions, and at the first level where there is a solution it explodes to a relatively large number.

\end{rem}

\begin{rem}\label{remark:other_results} Other solutions are available in the above-mentioned \href{https://github.com/leo-nc-git/perfectoid_kernel}{GitHub repository}. Namely, for $p=3$, we also find solutions for $n=2$ at $M=5$, $L=4$. They are so large in terms of the number of non-zero coefficients that it would be impractical to include them in this paper. 

We did not find any solution for $n>1$ for primes $p>3$. For $p=2$ and $p=3$, no solutions were found for $n>2$. For $p=2$, we tested up to $M=8$ and $L=7$ for $n=3$, where the calculations already took hours and no solution was found.
\end{rem}

We also explain that no finite approximation of some $\alpha$ as in \Cref{section:algori}  and \Cref{prop:kernel_results_2_2} can ever be in $(t-1)_{\pfd}$ already: a non-trivial infinite power series is really needed.

\begin{proposition}\label{prop:no_finite_approximation} Let $R$ be a perfectoid $p$-torsion-free $\Z_p^{\cycl}$-algebra. Identify as a module
\[R\abracket{t^{1/p^{\infty}}}/(t-1)\to \widehat{\bigoplus}_{\gamma\in \Z[1/p]\cap [0,1[} Rt^{\gamma}\]
to represent elements in this quotient. Then there is no finite sum in the image of $(t-1)_{\pfd}$.
\end{proposition}
\begin{proof}Suppose that there is such a polynomial $f\in R[t^{1/p^n}]$ for some fixed $n\in \N$. Note that we have the following commutative diagram
\[\begin{tikzcd}
{R[t^{1/p^n}]/(t-1)} \arrow[d] \arrow[r]     & {\Cacst{p}(\Z/p^n\Z,R)} \arrow[d] \\
R\abracket{t^{1/p^{\infty}}}/(t-1) \arrow[r] & {\Cacst{p}(\Z_p,R)}              
\end{tikzcd}\]
where the horizontal arrows are the perfectoidization maps (see \cite[Theorem 3.9, Proposition 3.18]{ishizuka2026perfectoidizaitonfinitealgebrasperfectoid}). The right vertical arrow is injective -- the top horizontal arrow is also injective because those rings are $p$-torsion-free and the perfectoidization map is a $p$-isogeny in this case by \cite[Theorem 10.9]{bhatt2022Prismsa}. 

If such an $f$ as above existed, then it would go to zero along those injective maps, implying that $f$ is zero already in the quotient, as desired.
\end{proof}

\nocite{*}
\printbibliography

@article{bhatt2022Prismsa,
  title = {Prisms and Prismatic Cohomology},
  author = {Bhatt, Bhargav and Scholze, Peter},
  year = {2022},
  journal = {Annals of Mathematics},
  volume = {196},
  number = {3},
  pages = {1135--1275},
  publisher = {Department of Mathematics of Princeton University},
  url = {https://projecteuclid.org/journals/annals-of-mathematics/volume-196/issue-3/Prisms-and-prismatic-cohomology/10.4007/annals.2022.196.3.5.full}
}

@misc{drinfeld2021notionringgroupoid,
      title={On a notion of ring groupoid}, 
      author={Vladimir Drinfeld},
      year={2021},
      eprint={2104.07090},
      archivePrefix={arXiv},
      primaryClass={math.CT},
      url={https://arxiv.org/abs/2104.07090}, 
}

@misc{antieau2024ktheorymathbfzpn,
      title={On the $K$-theory of $\mathbf{Z}/p^n$}, 
      author={Benjamin Antieau and Achim Krause and Thomas Nikolaus},
      year={2024},
      eprint={2405.04329},
      archivePrefix={arXiv},
      primaryClass={math.KT},
      url={https://arxiv.org/abs/2405.04329}, 
}

@article{BSMF_1978__106__225_0,
     author = {Fresnel, Jean and De Mathan, Bernard},
     title = {Alg\`ebres $L^1$ $p$-adiques},
     journal = {Bulletin de la Soci\'et\'e Math\'ematique de France},
     pages = {225--260},
     year = {1978},
     publisher = {Soci\'et\'e math\'ematique de France},
     volume = {106},
     doi = {10.24033/bsmf.1871},
     mrnumber = {80d:12014},
     zbl = {0392.46050},
     language = {fr},
     url = {https://www.numdam.org/articles/10.24033/bsmf.1871/}
}

@article{Scholze2015,
  author  = {Scholze, Peter},
  title   = {On torsion in the cohomology of locally symmetric varieties},
  journal = {Annals of Mathematics},
  year    = {2015},
  volume  = {182},
  number  = {3},
  pages   = {945--1066},
  doi     = {10.4007/annals.2015.182.3.3}
}

@book{FarguesFontaine2018,
  title     = {Courbes et fibrés vectoriels en théorie de Hodge $p$-adique},
  author    = {Fargues, Laurent and Fontaine, Jean-Marc},
  series    = {Astérisque},
  volume    = {406},
  publisher = {Société Mathématique de France},
  year      = {2018},
  isbn      = {978-2-85629-896-1},
  url       = {https://emath.fr}
}

@misc{bhatt2025Aspects,
  author = {Bhatt, Bhargav},
  title  = {Aspects of {$p$}-Adic {H}odge {T}heory},
  year   = {2025},
  url    = {https://www.math.ias.edu/~bhatt/teaching/mat517f25/pHT-notes.pdf},
  note   = {Lecture notes}
}

@article{cesnavicius2024purity,
  author = {{\v C}esnavi{\v c}ius, K{\k e}stutis and Scholze, Peter},
  title    = {Purity for flat cohomology},
  journal  = {Annals of Mathematics},
  volume   = {199},
  number   = {1},
  pages    = {51--182},
  year     = {2024},
  doi      = {10.4007/annals.2024.199.1.2},
  url      = {https://annals.math.princeton.edu/2024/199-1/p02}
}

@article{BMS18,
  author   = {Bhatt, Bhargav and Morrow, Matthew and Scholze, Peter},
  title    = {Integral {$p$}-adic {H}odge theory},
  journal  = {Publications math{\'e}matiques de l'IH{\'E}S},
  volume   = {128},
  pages    = {219--397},
  year     = {2018},
  doi      = {10.1007/s10240-019-00102-z},
  note     = {MR 3905467, Zbl 1446.14011}
}

@misc{ishizuka2026decompositiontheoremperfectoidrings,
      title={Decomposition Theorem for Perfectoid Rings along General Ideals}, 
      author = {Ishizuka, Ryo and Navarro Chafloque, L{\'e}o},
      year={2026},
      eprint={2606.06241},
      archivePrefix={arXiv},
      primaryClass={math.AC},
      url={https://arxiv.org/abs/2606.06241}, 
}

@misc{ishizuka2026perfectoidizaitonfinitealgebrasperfectoid,
      title={On Perfectoidizaiton of Finite Algebras over a Perfectoid Ring}, 
      author={Ishizuka, Ryo and Navarro Chafloque, L{\'e}o},
      year={2026},
      eprint={2606.12229},
      archivePrefix={arXiv},
      primaryClass={math.AC},
      url={https://arxiv.org/abs/2606.12229}, 
}

@article{Dine2024,
  author  = {Dine, Dimitri and Ishizuka, Ryo},
  title   = {Tilting and untilting for ideals in perfectoid rings},
  journal = {Mathematische Zeitschrift},
  year    = {2024},
  volume  = {307},
  number  = {4},
  pages   = {66},
  doi     = {10.1007/s00209-024-03537-1}
}

@misc{bhatt2022prismatizationpadicformalschemes,
      title={The prismatization of $p$-adic formal schemes}, 
      author={Bhargav Bhatt and Jacob Lurie},
      year={2022},
      eprint={2201.06124},
      archivePrefix={arXiv},
      primaryClass={math.AG},
      url={https://arxiv.org/abs/2201.06124}, 
}

@misc{stacks-project,
  author       = {The {Stacks project authors}},
  title        = {The Stacks project},
  howpublished = {\url{https://stacks.math.columbia.edu}},
  year         = {2026},
}

@misc{gabber2018foundationsringtheory,
      title={Foundations for almost ring theory -- Release 7.5}, 
      author={Ofer Gabber and Lorenzo Ramero},
      year={2018},
      eprint={math/0409584},
      archivePrefix={arXiv},
      primaryClass={math.AG},
      url={https://arxiv.org/abs/math/0409584}, 
}
\end{document}